\documentclass[runningheads]{llncs}

\usepackage{authblk}
\usepackage{stmaryrd}

\providecommand{\keywords}[1]{\smallskip\noindent\textbf{\textit{Keywords:}} #1}

\usepackage{graphicx}
\usepackage{tikz}
\usetikzlibrary{fit,patterns,decorations.pathmorphing,decorations.pathreplacing,calc,arrows}

\usepackage{amsmath,amssymb}
\usepackage[inline]{enumitem}

\usepackage{thmtools}

\usepackage[hidelinks]{hyperref}
\usepackage[capitalize,nameinlink,sort]{cleveref}

\usepackage{etoolbox}%

\usepackage[T1]{fontenc}
\usepackage{newpxtext}

\usepackage[textsize=scriptsize]{todonotes}

\DeclareMathOperator{\Fac}{Fac} 
\DeclareMathOperator{\N}{\mathbb{N}}

\DeclareMathOperator{\Sp}{Sp}
\newcommand{\infw}[1]{\mathbf{#1}}

\newcommand{\bc}[2]{
	\expandafter\ifstrequal\expandafter{#2}{1}{%
		 \mathsf{a}_{#1} }{%
		 \mathsf{b}_{#1}^{(#2)} }
             }
             
\makeatletter
\providecommand*{\shuffle}{%
  \mathbin{\mathpalette\shuffle@{}}%
}
\newcommand*{\shuffle@}[2]{%
  \sbox0{$#1\vcenter{}$}%
  \kern .15\ht0 
  \rlap{\vrule height .25\ht0 depth 0pt width 2.5\ht0}%
  \raise.1\ht0\hbox to 2.5\ht0{%
    \vrule height 1.75\ht0 depth -.1\ht0 width .17\ht0 %
    \hfill
    \vrule height 1.75\ht0 depth -.1\ht0 width .17\ht0 %
    \hfill
    \vrule height 1.75\ht0 depth -.1\ht0 width .17\ht0 %
  }%
  \kern .15\ht0 
}
\makeatother

\begin{document}

\title{Binomial Complexity of Multidimensional Arrays}
%
%
\author{Mehdi Golafshan \and
Michel Rigo\thanks{Supported by the FNRS Research grant T.196.23 (PDR)}\orcidID{0000-0001-7463-8507}}
\authorrunning{M. Golafshan, M. Rigo}

\institute{Univ. of Li\`ege, Dept. of Math., All\'ee de la d\'ecouverte 12 (B37), B-4000 Li\`ege, Belgium
  \email{
    mgolafshan@uliege.be, M.Rigo@uliege.be}}

\maketitle

\begin{abstract}
  Binomial coefficients for multidimensional arrays count occurrences of particular configurations. For the sake of presentation, the emphasis is put on column-binomial coefficients for two-dimensional finite arrays. In this article, we first show that these coefficients can be computed through some Magnus transform. 

  To get structural and combinatorial information on arrays, we then define \((k,\ell)\)-binomial equivalence for finite arrays and, from it, the \((k,\ell)\)-binomial complexity function of an infinite array. Roughly, two finite arrays are \((k,\ell)\)-binomially equivalent when they share the same number of subarrays of size at most $k\times \ell$.  

We obtain general results on the \((k,\ell)\)-binomial complexity of infinite arrays coding direct products of infinite words. Our main theorem gives an exact formula for the \((k,\ell)\)-binomial complexity of the two-dimensional Thue--Morse array. To that end, we closely examine the action of the bit-wise complement on the \(k\)-binomial equivalence classes of the factors of the Thue--Morse word.
\end{abstract}


\section{Introduction}

Let $\mathcal{A}$ be a finite alphabet. Recall that the {\em binomial coefficient} of two words $u=a_1\cdots a_k$ and $v=b_1\cdots b_\ell$ (where $a_i,b_j$  are letters in $\mathcal{A}$) is defined by
\begin{align}
  \binom{u}{v}&=\# \{1\le i_1<\cdots <i_{\ell}\le k \mid a_{i_1}\cdots a_{i_{\ell}}=v\}\label{eq:bin} \\
  &= \sum_{1\le i_1<\cdots <i_\ell\le k} \delta_{a_{i_1},b_1}\cdots \delta_{a_{i_\ell},b_\ell}.\nonumber
\end{align}
 This function counts the number of times $v$ appears as a subword, i.e., a subsequence, of $u$. For further details, see, for instance,  \cite[Chap.~6]{Lothaire1997}.

    The Kronecker delta can be thought of as a zero-dimensional binomial coefficient. Roughly speaking, letters are $0$-dimensional objects corresponding to points and words are $1$-dimensional objects corresponding to segments. Such a trivial observation permits us to introduce our object of study in a natural way. As binomial coefficients of words \eqref{eq:bin} are built from $0$-dimensional coefficients, $k$-dimensional coefficients are built from $(k-1)$-dimensional ones. The following definition of binomial coefficients of multidimensional arrays was first considered in \cite{GolafshanRigo}.

    Let $k\ge 1$ and $M$ be a {\em $k$-dimensional array} over $\mathcal{A}$ with dimension $(d_1,\ldots,d_k)$, i.e., a function from $\{1,\ldots,d_1\}\times\cdots\times\{1,\ldots,d_k\}$ to $\mathcal{A}$. For each $j\in\{1,\ldots,d_k\}$, we let $M_j$ denote the $(k-1)$-dimensional {\em section} (or slice) of $M$ where the last component is set to $j$. Hence, $M_j$ is a $(k-1)$-dimensional array of dimension $(d_1,\ldots,d_{k-1})$ where
$$M_j[i_1,\ldots,i_{k-1}]=M[i_1,\ldots,i_{k-1},j].$$
In particular, if $k=1$, then $M_j$ denotes the $j^{\text{th}}$ letter of the word $M$ ; if $k=2$, then $M_j$ is the $j^{\text{th}}$ column of the two-dimensional array $M$.

\begin{definition}[{\cite[Def.~2]{GolafshanRigo}}]
Let $M \in \mathcal{A}^{d_1\times \cdots \times d_k}$ and $P\in \mathcal{A}^{e_1\times \cdots \times e_k}$ be $k$-dimensional arrays.  
The {\em binomial coefficient} of $M$ and $P$ is given by
\[
  \binom{M}{P}=\sum_{1\le i_1< \cdots <i_{e_k}\le d_k}\binom{M_{i_1}}{P_1}\cdots\binom{M_{i_{e_k}}}{P_{e_k}}
\]
where, on the r.h.s., we have  binomial coefficients of $(k-1)$-dimensional section arrays.
\end{definition}

By convention, \(\binom{M}{P}=0\) whenever $d_j<e_j$ for some $j$ and \(\binom{M}{\varepsilon}=1\) for an empty array.

\begin{remark}\label{rem:order}
  The order in which consecutive sections are recursively performed is important. For convenience, we have chosen to always fix the last component to reduce the dimension by one. This choice is arbitrary: any other ordering among the $k!$ available ones would lead to a different counting function (but with similar properties). This dependence on the ordering may be seen as a limitation of our definition. However, no satisfactory generalization of the notion of subsequence exists in a multidimensional setting. For instance in two dimensions, one may think of various alternatives: submatrices obtained by selecting a subset of rows and columns, submatrices made of contiguous rows and columns, connected shapes such as paths or polyominoes.

  Our aim is to preserve classical combinatorial identities such as Pascal's identity recalled below, see \cref{pro:pascal}. A justification is given in this article by \cref{pro:magnus} where it is shown that these coefficients can be computed by a convenient Magnus transformation. It is indeed classical to define binomial coefficients of words from the Magnus transformation, see \cite{Reutenauer}. We show that these coefficients count subarrays of a particular type, see \cref{def:subarray} where constant-length subsequences are extracted from a selection of columns.
\end{remark}

For the sake of presentation, we consider two-dimensional arrays and coefficients. In that case, we have the following definition. Although the concepts we  develop apply to both the rows and columns,
 we focus on columns for concise notation; analogous definitions for rows are straightforward.

\begin{definition}\label{def:main}
Let $ M=
  \begin{pmatrix}
    M_1& \cdots &M_m
  \end{pmatrix}$
and 
$
  P=
  \begin{pmatrix}
    P_1& \cdots &P_p
  \end{pmatrix}
$ be two-dimensional arrays. We define the {\em column-binomial coefficient}, simply the {\em binomial coefficient}, of the arrays $M$ and $P$ by
\[
  \binom{M}{P}=\sum_{1\le j_1< \cdots <j_p\le m}\binom{M_{j_1}}{P_1}\cdots\binom{M_{j_p}}{P_p}
\]
where on the r.h.s. we have classical binomial coefficients of (unidimensional) words, with columns $M_{j_i}$ and $P_i$  interpreted as words.  
\end{definition}
In particular, if $m<p$ or if the number of rows of $M$ is less than that of $P$, then $\binom{M}{P}$ is zero. Note that when $m=p=1$, this reduces to the usual binomial coefficient of words. 

\subsection{Our contributions}

After recalling in \cref{sec:prelim} basic results from \cite{GolafshanRigo} on binomial coefficients for arrays, we introduce in \cref{sec:magnus} the two-dimensional Magnus transform of an array~$M$. It is a formal series in $\mathbb{N}\langle\langle\mathcal{A}^{**}\rangle\rangle$ such that
\[
  \mu(M)=\sum_{P\in \mathcal{A}^{**}} \binom{M}{P} P.
\]
The coefficient of $P$ in $\mu(M)$ counts, as expected, the number of times $P$ appears as a particular subarray of $M$ where constant-length subsequences are extracted from a selection of columns.

In \cref{sec:binom}, we introduce the main definitions of \((k,\ell)\)-binomial equivalence for finite arrays and, from it, the \((k,\ell)\)-binomial complexity $\bc{\infw{M}}{k,\ell}$ of an infinite array~$\infw{M}$. This function captures information about the combinatorial structure and complexity of the array and its configurations. 
In \cref{sec:bound}, we show that if $\infw{M}$ is the direct product $\infw{s}\times\infw{y}$ of two infinite words, then $\bc{\infw{M}}{k,\ell}(m,n)=\bc{\infw{s}}{k}(m).\bc{\infw{y}}{\ell}(n)$. \cref{thm:bounded_complexity} studies a more general situation than direct products.

 Let $k\ge 1$, $\ell\ge 2$ and $m,n\ge 1$. \cref{thm:mainTM} gives the exact formula of the $(k,\ell)$-binomial complexity of the two-dimensional Thue--Morse word~$\infw{T}$, an example of $2$-automatic array \cite{Salon}, in terms of the binomial complexity of the Thue--Morse word $\infw{t}=01101001\cdots$:
\[
  \bc{\infw{T}}{k,\ell}(m,n)=\left\{\begin{array}{ll}
   \frac{\bc{\infw{t}}{k}(m).\bc{\infw{t}}{\ell}(n) }{2}, & \text{ if } m\not\equiv 0\pmod{2^k};\\
        \frac{(\bc{\infw{t}}{k}(m)-1).\bc{\infw{t}}{\ell}(n)}{2}+1, & \text{ if }m\equiv 0\pmod{2^k}.\\
      \end{array}\right.
  \]
In particular, for fixed $k\ge 1$ and $\ell\ge2$, the function $\bc{\infw{T}}{k,\ell}(m,n)$ is ultimately periodic in each coordinate, with periods $2^k$ and $2^\ell$, respectively. To get that result, we develop new tools and techniques related to $k$-binomial equivalence of words. Consider the quotient of the set of factors of $\infw{t}$ by the $k$-binomial equivalence relation. In \cref{sec:bitwise}, we show that the bit-wise complement maps every word of a $k$-binomial equivalence class to a word of another class and conversely, or it fixes a class. If $2^k|n$, then exactly one $k$-binomial equivalence class is fixed by bit-wise complement. Otherwise, no equivalence class is fixed. 
\subsection{Related works}

\cref{sec:bitwise} complements the study of the binomial complexity of the Thue--Morse word initiated in \cite{LejeuneLeroyRigo-2020}. See also \cite{GRW,RigoStipulantiWhiteland-2024,ChenWen-2019,ChenWen-2024} for generalizations to larger alphabets. Indeed, the Thue--Morse word has attracted a lot of attention and is a main object of study in combinatorics, see also \cite{Starosta-2012}. Generalized Thue--Morse words are automatic sequences and thus have linear factor complexity, while their $k$-binomial complexity is ultimately periodic with a period depending on $k$ (hence, bounded by a constant). See \cref{thm:lejeune}.

In higher dimension, the two-dimensional Thue--Morse array is one of the simplest examples of an aperiodic tiling. This has applications and connects directly to the theory of quasicrystals, ergodic properties, diffraction spectra, and substitution systems. See, for instance, \cite{Baake,Colle-2023,Li2025,Robinson}.

Finally, the study of repetitive configurations and, in particular periodicity, plays a prominent role and $(k,\ell)$-binomial complexity gives some structural information on the encountered configurations in arrays \cite{QuasZamboni}. Puzynina proves an abelian analogue of Nivat’s conjecture \cite{Kra,KariSzabados} for recurrent two-dimensional arrays \cite{Puzynina-2019}. If, for every $m,n\ge 1$, the set of $m\times n$ factors has at most two abelian classes, i.e., the $(1,1)$-binomial complexity is bounded by $2$, then the infinite array has a periodicity vector. Conversely, the  abelian complexity of any aperiodic recurrent array must exceed $2$ for infinitely many sizes. She also shows that abelian complexity $1$ can occur at some sizes even for aperiodic arrays (with a precise lattice periodicity), and constructs recurrent aperiodic examples with abelian complexity uniformly bounded by $3$.

\section{Preliminaries}\label{sec:prelim}

Alphabets are denoted by calligraphic letters, e.g., $\mathcal{A}$. The set of finite one-dimensional (resp. two-dimensional) words is denoted by $\mathcal{A}^*$ (resp. $\mathcal{A}^{**}$).  The set of infinite one-dimensional (resp. two-dimensional) words is denoted by $\mathcal{A}^\mathbb{N}$ (resp. $\mathcal{A}^{\mathbb{N}\times\mathbb{N}}$). We write $\mathcal{A}^{\le \ell}$ (resp. $\mathcal{A}^\ell$) for the set of words of length at most (resp. exactly)~$\ell$ ; and $\mathcal{A}^{m\times \ell}$ denotes the set of $m\times \ell$ arrays. 

Finite one-dimensional words are denoted by lowercase letters, while infinite words are denoted by boldface lowercase letters, e.g., $u,v \in\mathcal{A}^*$ and $\infw{x},\infw{t}\in\mathcal{A}^\mathbb{N}$. Finite two-dimensional arrays are denoted by uppercase letters and infinite two-dimensional arrays by their boldface counterparts, e.g., $U,M\in\mathcal{A}^{**}$ and $\infw{U}, \infw{M}\in\mathcal{A}^{\mathbb{N}\times\mathbb{N}}$.  The one-dimensional Thue–Morse morphism (resp. word) is denoted by $\varphi$ (resp. $\infw{t}$), and its two-dimensional analogue by the uppercase Greek letter~$\Phi$ (resp. $\infw{T}$).

The horizontal concatenation of two arrays is defined, from left to right, when they have the same number of rows. Similarly, we define vertical concatenation, from top to bottom. This choice is consistent with the usual convention that rows and columns are indexed starting from the upper-left corner. 

\begin{definition}
Let $U = (u_{i,j})\in \mathcal{A}^{m \times n}$ and $V = (v_{i,j})\in \mathcal{A}^{m \times p}$ be two arrays over an alphabet $\mathcal{A}$. The \emph{horizontal concatenation} of $U$ and $V$, denoted $U \obar V$, 
is the array $W \in \mathcal{A}^{m \times (n+p)}$ defined by
\[
W_{i,j} = 
\begin{cases}
u_{i,j}, & 1 \leq j \leq n, \\
v_{i,\,j-n}, & n+1 \leq j \leq n+p,
\end{cases}
\qquad 1 \leq i \leq m.
\]
Similarly, if $U = (u_{i,j})\in \mathcal{A}^{m \times n}$ and $V = (v_{i,j})\in \mathcal{A}^{p \times n}$ are two arrays with the same number of columns, then the \emph{vertical concatenation} of $U$ and $V$, denoted $U\ominus V$
is the array $W \in \mathcal{A}^{(m+p) \times n}$ defined by
\[
W_{i,j} = 
\begin{cases}
u_{i,j}, & 1 \leq i \leq m, \\
v_{i-m,\,j}, & m+1 \leq i \leq m+p,
\end{cases}
\qquad 1 \leq j \leq n.
\]
\end{definition}

Since we are dealing with column-binomial coefficients, see \cref{def:main}, we have to consider horizontal concatenation to get a Pascal-type formula. We recall some results from \cite{GolafshanRigo}.
\begin{proposition}[Pascal formula {\cite[Prop.~1]{GolafshanRigo}}]\label{pro:pascal}
Let $M \in \mathcal{A}^{r\times m}$, $P\in \mathcal{A}^{s\times p}$, $C\in \mathcal{A}^{r\times 1}$ and $D\in \mathcal{A}^{s\times 1}$. 
\begin{equation}\label{eq:pascal}
  \binom{M\obar C}{P\obar D}=\binom{M}{P\obar D}+ \binom{C}{D} \binom{M}{P}.
\end{equation}
\end{proposition}

An alternative proof of the next proposition is given as a consequence of \cref{pro:magnus} on Magnus transform.
\begin{proposition}[Chu--Vandermonde identity {\cite[Prop.~2]{GolafshanRigo}}]\label{pro:vandermonde}
Let $M \in \mathcal{A}^{r\times m}, N\in \mathcal{A}^{r\times n}$ so that $M\obar N$ belongs to $\mathcal{A}^{r\times (m+n)}$
\[
  \binom{M\obar N}{P}=\sum_{P_1\obar P_2=P}\binom{M}{P_1}\binom{N}{P_2}.
\]
\end{proposition}

\begin{proposition}[{\cite[Prop.~3]{GolafshanRigo}}]\label{pro:count_pos}
Let $M \in \mathcal{A}^{r\times m}$
$$\sum_{P\in \mathcal{A}^{s\times p}} \binom{M}{P}=\binom{a^{r\times m}}{a^{s\times p}}= \binom{m}{p} \cdot\binom{r}{s}^p.$$
\end{proposition}

For words, the following formula was first proposed in \cite{ManvelMSSS1991reconstruction}.

\begin{theorem}[Manvel-type formula {\cite[Thm.~2]{GolafshanRigo}}]\label{thm:manvel_2D}
Let $M\in \mathcal{A}^{r\times m}$, $P\in \mathcal{A}^{s\times p}$ and  $r\ge q\ge s$, $m\ge t\ge p$. We have 
\[
\sum_{T\in \mathcal{A}^{q\times t}} \binom{M}{T}\binom{T}{P}=\binom{m-p}{t-p} \binom{r}{q}^{t-p} \binom{r-s}{q-s}^p \binom{M}{P}.
\]
\end{theorem}

This corollary will be useful when dealing with binomial equivalence.

\begin{corollary}[Manvel-type formula {\cite[Cor.~1]{GolafshanRigo}}]\label{cor:manvel}
    Let $M,M'\in \mathcal{A}^{r\times m}$. Let $q,t,s,p$ be such that $r\ge q\ge s$ and $m\ge t\ge p$. If $\binom{M}{T}=\binom{M'}{T}$ for all $T\in \mathcal{A}^{q\times t}$, then $\binom{M}{P}=\binom{M'}{P}$ for all $P\in \mathcal{A}^{s\times p}$.
\end{corollary}

\section{Magnus transformation}\label{sec:magnus}

Recall that the horizontal (resp. vertical) partial concatenation $\obar$ (resp. $\ominus$) is defined for arrays with the same number of rows (resp. columns). We will use the following conventions: if the operation is undefined, then the value of a forbidden product is set to zero. Also, an empty array acts as a neutral element for both partial concatenations. 

A {\em formal series} over a monoid $\mathcal{M}$ with coefficients in a semiring~$\mathbb{K}$ is simply a map $\mathsf{s}:\mathcal{M}\to\mathbb{K}$ and is often written as a sum $\sum_{m\in\mathcal{M}}\mathsf{s}(m)\, m$. The set of these series is denoted by $\mathbb{K}\langle\langle \mathcal{M}\rangle\rangle$. A {\em polynomial} is a series with finite support. The coefficient of $m$ in $\mathsf{s}$ is usually denoted by $\mathsf{s}(m)=\langle\mathsf{s},m\rangle$.

The transformation of a column $C=\begin{pmatrix}
  a_1& \cdots &a_m
\end{pmatrix}^\top$, where $a_i\in\mathcal{A}$ are letters, by $\mu_{\ominus}$ is defined, as in the classical case of one-dimensional words, by
\[
  \mu_{\ominus}(C)=(1+a_1)\ominus \cdots \ominus(1+a_m).
\]
It is a formal polynomial in $\mathbb{N}\langle\langle \mathcal{A}^*\rangle\rangle$ associating non-negative integers with column vectors of size at most $m$. As an example, apply $\mu_{\ominus}$ separately to the two columns of
\[
  \begin{pmatrix}
    a&b\\b&a\\a&a\\
  \end{pmatrix}.
\]
We obtain the two formal polynomials 
\[
  \mathsf{p}=(1+a)\ominus (1+b)\ominus (1+a)= 1+2
  \begin{pmatrix}
    a
  \end{pmatrix}+
  \begin{pmatrix}
    b
  \end{pmatrix}+
  \begin{pmatrix}
    a\\ b\\
  \end{pmatrix}+
  \begin{pmatrix}
    a\\ a\\
  \end{pmatrix}+
  \begin{pmatrix}
    b\\ a\\
  \end{pmatrix}+
  \begin{pmatrix}
    a\\b\\a\\
  \end{pmatrix}.
\]
and 
\[
   \mathsf{q}=(1+b)\ominus (1+a)\ominus (1+a)= 1+2
  \begin{pmatrix}
    a
  \end{pmatrix}+
  \begin{pmatrix}
    b
  \end{pmatrix}+
  \begin{pmatrix}
    a\\ a\\
  \end{pmatrix}+2
  \begin{pmatrix}
    b\\ a\\
  \end{pmatrix}+
  \begin{pmatrix}
    b\\a\\a\\
  \end{pmatrix}.
\]
Actually, we have the following result \cite[Prop.~6.3.6]{Lothaire1997}.
\begin{proposition}\label{pro:magnus1D}
The Magnus transformation of the column $C=\begin{pmatrix}
  a_1& \cdots &a_m
\end{pmatrix}^\top$, where $a_i\in\mathcal{A}$, is
\[
  \mu_{\ominus}(C)=\sum_{w\in\mathcal{A}^*} \binom{C}{w}\, w.
\]
\end{proposition}
If we have two formal polynomials $\mathsf{s},\mathsf{t}\in\mathbb{N}\langle\langle \mathcal{A}^{**}\rangle\rangle$, we define $\mathsf{s}\obar \mathsf{t}$ as in \cite{Maurer} (where formal series on pictures are discussed) by 
\[
  \langle \mathsf{s}\obar\mathsf{t}, U\rangle = \sum_{U_1\obar U_2=U} \langle \mathsf{s}, U_1\rangle .\langle \mathsf{t}, U_2\rangle.
\]
This is where the convention about undefined product matters: If two non-empty arrays have different heights, their horizontal concatenation is the zero polynomial. The empty array acts as the identity. 
As an example, $\mathsf{p}\obar\mathsf{q}$ is equal to
\begin{align*}
   &1+4
  \begin{pmatrix}
    a
  \end{pmatrix}+2
  \begin{pmatrix}
    b
  \end{pmatrix}+
  \begin{pmatrix}
    a\\ b\\
  \end{pmatrix}+2
  \begin{pmatrix}
    a\\ a\\
  \end{pmatrix}+3
  \begin{pmatrix}
    b\\ a\\
  \end{pmatrix}+
  \begin{pmatrix}
    a\\b\\a\\
  \end{pmatrix}+
  \begin{pmatrix}
    b\\a\\a\\
  \end{pmatrix}\\
  &+4(aa)+2(ab)+2(ba)+(bb)+
  \begin{pmatrix}
    a&a\\b&a
  \end{pmatrix}+
            \begin{pmatrix}
              a&a\\a&a\\
            \end{pmatrix}+
  \begin{pmatrix}
    b&a\\a&a\\
  \end{pmatrix}\\
  &+ 2\begin{pmatrix}
    a&b\\b&a\\
  \end{pmatrix}+
  2\begin{pmatrix}
    a&b\\a&a\\
  \end{pmatrix}+
  2\begin{pmatrix}
    b&b\\a&a\\
  \end{pmatrix}+ \begin{pmatrix}
    a&b\\b&a\\a&a\\
  \end{pmatrix}.
\end{align*}
Note that the first line is essentially $\mathsf{p}+\mathsf{q}$. 

The following definition is also order dependent as is \cref{def:main}. 

\begin{definition}
  The two-dimensional \emph{Magnus transformation} of an array $M\in\mathcal{A}^{k\times \ell}$ is
  \[
    \mu(M):=\obar_{j=1}^\ell \left( \ominus_{i=1}^k (1+M[i,j])\right).
  \]
  It is a formal polynomial in $\mathbb{N}\langle\langle \mathcal{A}^{**}\rangle\rangle$.
\end{definition}
It is straightforward to extend this definition to higher dimension. As in \cref{rem:order}, the order in which the directional concatenations are applied matters.
\begin{lemma}
  The two-dimensional \emph{Magnus transformation} is a morphism, i.e., for all $M\in\mathcal{A}^{k\times \ell}$ and $N\in\mathcal{A}^{k\times r}$, $\mu(M\obar N)=\mu(M)\obar \mu(N)$.
\end{lemma}

We therefore obtain a Newton's formula. The classical Newton's formula develops $(1+x)^n$ and the coefficient of $x^j$ is the binomial coefficient $\binom{n}{j}$ counting the number of possible ways to choose a $j$-subset from an $n$-set. Here, we define subarray replacing the $j$-subsets from the classical setting.

\begin{definition}\label{def:subarray}
  Let $M \in \mathcal{A}^{r\times m}$ and $N\in \mathcal{A}^{s\times p}$ with $s\le r$ and $p\le m$. We say that $N$ is a {\em subarray} of $M$ if there exist $1\le j_1<\cdots <j_p\le m$ such that $N_i$ is a subword of $M_{j_i}$ for all $i\in\{1,\ldots,p\}$.
\end{definition}

One can therefore introduce a {\em subarray partial order} on $\mathcal{A}^{**}$.

\begin{proposition}\label{pro:magnus}
  The coefficients of the two-dimensional Magnus transform of an array $M$ are exactly the binomial coefficients of the form $\binom{M}{P}$, i.e.,  
  \[
    \mu(M)=\sum_{P\in \mathcal{A}^{**}} \binom{M}{P} P.
  \]
\end{proposition}

\begin{proof}
  Let $M=
  \begin{pmatrix}
    M_1&\cdots & M_\ell
  \end{pmatrix}$. By definition, we have
  \[
    \mu(M)=\obar_{j=1}^\ell \mu_\ominus(M_j).
  \]
  Since $\obar$-concatenation is defined only when all factors have the same number of rows (otherwise, the result is $0$) or are equal to an empty array, any non-zero term of $\mu(M)$ with a positive number~$t$ of rows is obtained by picking $k$ columns with $t$ rows in $\mu_\ominus(M_{i_1}),\ldots,\mu_\ominus(M_{i_k})$ and $1$ in the $\ell-k$ other terms. Let $P=\begin{pmatrix}
    P_1&\cdots & P_k
  \end{pmatrix}\in\mathcal{A}^{t\times k}$. We get
  \begin{align*}
    \langle \mu(M), P\rangle &= \sum_{1\le i_1<\cdots <i_k\le\ell} \langle\mu_\ominus(M_{i_1}), P_1\rangle \cdots \langle\mu_\ominus(M_{i_k}), P_k\rangle\\
    &=\sum_{1\le i_1<\cdots <i_k\le\ell} \binom{M_{i_1}}{P_1}\cdots \binom{M_{i_k}}{P_k}=\binom{M}{P}
    \end{align*}
  \qed
\end{proof}

As a Corollary, we get an alternative proof of Chu--Vandermonde identity.

\begin{proof}[of \cref{pro:vandermonde}]
Let $M \in \mathcal{A}^{k\times m}, N\in \mathcal{A}^{k\times n}$. We have
\begin{align*}
  \mu(M\obar N)&=\sum_{P\in \mathcal{A}^{**}} \binom{M \obar N}{P} P
\end{align*}
and since $\mu(M\obar N)=\mu(M)\obar\mu(N)$
\begin{align*}
  \mu(M\obar N)&= \left( \sum_{P_1\in \mathcal{A}^{**}} \binom{M}{P_1} P_1 \right) \obar \left( \sum_{P_2\in \mathcal{A}^{**}} \binom{N}{P_2} P_2 \right)\\
                 &= \sum_{P\in \mathcal{A}^{**}}  \sum_{P_1\obar P_2=P}\binom{M}{P_1}\binom{N}{P_2} P.
\end{align*}
  \qed
\end{proof}

Mimicking \cite{Manuch,Salomaa2003}, one can introduce the {\em $(k,\ell)$-spectrum} which is the formal polynomial encoding some binomial coefficients of $M$ for subarrays of specific size
\[
    \Sp_{k\times\ell}(M)=\sum_{P\in \mathcal{A}^{k\times \ell}} \binom{M}{P} P
  \]
  and the {\em full $(k,\ell)$-spectrum}
  which is the formal polynomial
\[
    \Sp_{\le k\times\ell}(M)=1+\sum_{\substack{1\le i\le k,\ 1\le j\le \ell\\ P\in \mathcal{A}^{i\times j}}} \binom{M}{P} P.
  \]
  
\section{Binomial complexity in a two-dimensional setting}\label{sec:binom}

We first recall the notions of $k$-binomial equivalence for finite words and $k$-binomial complexity for infinite words. These notions were first introduced in \cite{RigoSalimov-2015}. Then we extend those in a natural way.
\subsection{Recap in the one-dimensional case}
Let $k\ge 1$ be an integer. 
Two words $u, v\in \mathcal{A}^*$ are \emph{$k$-binomially equivalent},  and we write $u \sim_k v$,  if
\[
\binom{u}{w} = \binom{v}{w}
\]
for all $w\in \mathcal{A}^{\le k}$. The $k$-binomial equivalence class of $u$ is denoted by $[u]_k$. This equivalence relation was introduced in \cite{RigoSalimov-2015}. For a survey, see \cite{RSWCANT}. This is a congruence refining abelian equivalence:
\[
\forall u,v,x,y\in\mathcal{A}^*:\quad u\sim_k v, x\sim_k y\Rightarrow ux\sim_k vy.
\]

\begin{remark}
As a consequence of Manvel's formula \cite{ManvelMSSS1991reconstruction}, we may replace the condition on all $w\in \mathcal{A}^{\leq k}$ with the condition $\binom{u}{w} = \binom{v}{w}$ restricted to all $w\in \mathcal{A}^k$ as soon as $|u|$, $|v| \geq k$.  
\end{remark}

Let $\varphi:0\mapsto 01$, $1\mapsto 10$ be the Thue--Morse morphism. Ochsenschl\"ager showed \cite{Ochsenschlager-1981} that, for all $k\ge 1$,
\begin{equation}
  \label{eq:Oschenschlager}
  \varphi^k(0)\sim_k \varphi^k(1)\quad \text{ and }\quad  \varphi^k(0)\not\sim_{k+1} \varphi^k(1). 
\end{equation}

As an example, $\varphi^2(0)=0110\sim_2\varphi^2(1)=1001$ by direct computation
\[
  \binom{0110}{01}=2=\binom{1001}{01}
\]
and abelian equivalence alone does not imply $2$-binomial equivalence. However,  $\varphi^2(0)=0110\not\sim_3\varphi^2(1)=1001$ because
\[
  \binom{0110}{010}=2\neq 0=\binom{1001}{010}.
\]
If we exchange $0$'s and $1$'s, we see that $\binom{1001}{010}=\binom{0110}{101}$. The reader may already notice that bit-wise complement and the difference $\binom{0110}{010}-\binom{0110}{101}$ play here a special role to find non-equivalent words. See \cref{def:deltasigma} and \cref{eq:deltak1} for a formula about such differences.

\begin{definition}
We let $\Psi(u)$ denote the Parikh vector of the word~$u$. A morphism $f\colon \mathcal{A}^* \to \mathcal{B}^*$ is said to be \emph{Parikh-collinear}\index{Parikh-collinear} if, for all letters $a,b \in \mathcal{A}$, there exists $r_{a,b}\in\mathbb{Q}$ such that $\Psi(f(b)) = r_{a,b} \Psi(f(a))$. As a particular case,  $f\colon \mathcal{A}^* \to \mathcal{B}^*$ is said to be {\em Parikh-constant} if $\Psi(f(a)) = \Psi(f(b))$ for all $a,b \in \mathcal{A}$.
\end{definition}

The Thue--Morse morphism is a special case of Parikh-constant morphism, we recall some of its important properties. See \cite{RigoStipulantiWhiteland-2024,GRW}.

\begin{lemma}[{Transfer lemma \cite[Lem.~31]{LejeuneLeroyRigo-2020}}]\label{lem:transfer}
  Let $f$ be a Parikh-constant morphism and $u,v,w$ be words such that $|v|=|w|$. For all $k\ge 1$, we have $f^{k-1}(u)f^k(v)\sim_kf^k(w)f^{k-1}(u)$.
\end{lemma}

\begin{proposition}\label{pro:equivTM}
  Let $\varphi$ be the Thue--Morse morphism, $u,v\in\{0,1\}^*$ and $k\ge 1$, $u\sim_1 v\Leftrightarrow\varphi^k(u)\sim_{k+1}\varphi^k(v)$.
\end{proposition}

The fact that $u\sim_1 v\Rightarrow\varphi^k(u)\sim_{k+1}\varphi^k(v)$ holds for any Parikh-collinear morphism as shown in \cite[Prop.~3.9]{RigoStipulantiWhiteland-2024}. The converse holds for the Thue--Morse morphism (generalized over $m$ letters), see \cite[Prop.~2.3]{GRW}.

\begin{definition}[Binomial complexity]
Let $k\ge 1$ be an integer. 
The \emph{$k$-binomial complexity function} of an infinite word $\infw{x}\in\mathcal{A}^\mathbb{N}$ is defined by the map $\bc{\infw{x}}{k} \colon \mathbb{N} \to \mathbb{N}$, $n\mapsto \#(\Fac_n(\infw{x})/{\sim_k})$.
\end{definition}

In particular, if $n\ge k$, then the $k$-binomial complexity counts some extended Parikh vectors
\[
  \Psi_k(u):= \left( \binom{u}{w_1},\ldots,\binom{u}{w_{\# \mathcal{A}^k}} \right).
\]

\begin{equation}\label{eq:parikh}
  \bc{\infw{x}}{k}(n) =\#\left\{ \Psi_k(u) \mid u\in \Fac_n(\infw{x})\right\}
\end{equation}
where $w_1<\cdots <w_{\# \mathcal{A}^k}$ range over the words of $\mathcal{A}^k$ genealogically ordered.

\begin{theorem}[{\cite[Cor.~3.6]{RigoStipulantiWhiteland-2024}}]
  Let $\infw{x}$ be a fixed point of a Parikh-collinear morphism.
For any $k \geq 1$ there exists a constant $C_{\infw{x},k} \in \N$ such that
$\bc{\infw{x}}{k}(n)\leq C_{\infw{x},k}$ for all $n \in \N$.
\end{theorem}

\begin{theorem}[{\cite[Thm.~6]{LejeuneLeroyRigo-2020}}]\label{thm:lejeune}
  Let $k\ge 1$. For $n\ge 2^k$, the $k$-binomial complexity of the Thue--Morse word~$\infw{t}$ is given by 
  \[
    \bc{\infw{t}}{k}(n) =\left\{
      \begin{array}{ll}
        3.2^k-3, \text{ if }n\equiv 0\pmod{2^k};\\
        3.2^k-4, \text{ otherwise}.
      \end{array}\right.
  \]
  For $n<2^k$, $k$-binomial complexity and factor complexity coincide.
\end{theorem}

\subsection{Extension to a multidimensional setting}

We extend the definition to a two-dimensional setting in a straightforward manner.
\begin{definition}[Binomial equivalence]
Let $k,\ell\ge 1$ be integers. Two finite arrays $U,V$ of the same dimension are \emph{$(k,\ell)$-binomially equivalent}, and we write $U\sim_{k,\ell}V$, if
\[
  \binom{U}{W}=\binom{V}{W}
\]
for all arrays $W$ of dimension $k'\times\ell'$, for all $k'\le k$ and $\ell'\le \ell$. In terms of full spectrum, this equivalence can be expressed as 
\[
  \Sp_{\le k\times\ell}(U)=\Sp_{\le k\times\ell}(V).
\]
\end{definition}

\begin{remark}
  It is obvious that $U\sim_{k,\ell} V$ implies $U\sim_{k',\ell'} V$ for all $k'\le k$ and $\ell'\le\ell$. In particular, $\sim_{1,1}$ is the abelian equivalence.
\end{remark}

Assume that $U$ and $V$ are large enough: they both contain a subarray of dimension $k\times\ell$. As a consequence of \cref{cor:manvel}, we have the following.

\begin{proposition}\label{pro:limitkl}
  Let $U,V\in\mathcal{A}^{r\times m}$, $k\le r$ and $\ell\le m$. We have $U\sim_{k,\ell}V$ whenever $\binom{U}{W}=\binom{V}{W}$ holds for all arrays $W$ of dimension $k\times \ell$, i.e.,
  \[
    \Sp_{k\times\ell}(U)=\Sp_{k\times\ell}(V)\Rightarrow U\sim_{k,\ell}V.
  \]
\end{proposition}

Let $\infw{M}\in \mathcal{A}^{\mathbb{N}\times\mathbb{N}}$. For $m,n\ge 1$, we define the set of $(m,n)$-factors of $\infw{M}$ by
  \[\Fac_{m\times n}(\infw{M}):=\{ \infw{M}[i\cdots i+m-1,j\cdots j+n-1]\mid i,j\in \mathbb{N}\} .
  \]

\begin{definition}[Binomial complexity]
Let $k,\ell\ge 1$ be integers. 
The \emph{$(k,\ell)$-binomial complexity function} of an infinite array $\infw{M}$ is the map $\bc{\infw{M}}{k,\ell} \colon \mathbb{N}\setminus\{0\}\times\mathbb{N}\setminus\{0\}\ \to \mathbb{N}$, $(m,n)\mapsto \#(\Fac_{m\times n}(\infw{M})/{\sim_{k,\ell}})$.
\end{definition}


As an example, here are the eight $2\times 2$ factors occurring in the two-dimensional Thue--Morse word $\infw{T}$:
\[
  \left(
\begin{array}{cc}
 0 & 0 \\
 0 & 0 \\
\end{array}
\right),\left(
\begin{array}{cc}
 0 & 0 \\
 1 & 1 \\
\end{array}
\right),\left(
\begin{array}{cc}
 0 & 1 \\
 0 & 1 \\
\end{array}
\right),\left(
\begin{array}{cc}
 0 & 1 \\
 1 & 0 \\
\end{array}
\right),\left(
\begin{array}{cc}
 1 & 0 \\
 0 & 1 \\
\end{array}
\right),\left(
\begin{array}{cc}
 1 & 0 \\
 1 & 0 \\
\end{array}
\right),\left(
\begin{array}{cc}
 1 & 1 \\
 0 & 0 \\
\end{array}
\right),\left(
\begin{array}{cc}
 1 & 1 \\
 1 & 1 \\
\end{array}
\right)
\]

For instance, the following two $4\times 4$ factors of $\infw{T}$ satisfy
\[
  \left(
\begin{array}{cccc}
 1 & 1 & 0 & 1 \\
 0 & 0 & 1 & 0 \\
 0 & 0 & 1 & 0 \\
 1 & 1 & 0 & 1 \\
\end{array}
\right)\sim_{2,2}
\left(
\begin{array}{cccc}
 0 & 1 & 0 & 0 \\
 1 & 0 & 1 & 1 \\
 1 & 0 & 1 & 1 \\
 0 & 1 & 0 & 0 \\
\end{array}
\right)
\]
either compute the relevant binomial coefficients or use \cref{lem:colum-k-equiv} because $0110\sim_2 1001$.

\section{A general bound on the $(k,\ell)$-binomial complexity}\label{sec:bound}

\begin{definition}
  Let $u\in\mathcal{A}^*$ and $v\in\mathcal{B}^*$ (resp. $\infw{x}\in\mathcal{A}^\mathbb{N}$, $\infw{y}\in\mathcal{B}^\mathbb{N}$) be words. The {\em direct product} $u\times v$ (resp. $\infw{x}\times\infw{y}$) of these two words is an array in $(\mathcal{A}\times\mathcal{B})^{**}$ (resp. $(\mathcal{A}\times\mathcal{B})^{\mathbb{N}\times\mathbb{N}}$) whose element $i,j$ is $(u_i,v_j)$ (resp. $(\infw{x}_i,\infw{y}_j)$).
  \[
    \begin{array}{c|cccccc}
      \times &b_1&b_2&\cdots &b_n\\
      \hline
      a_1 & (a_1,b_1) & (a_1,b_2) & \cdots &(a_1,b_n) \\
      \vdots& \vdots & & & \vdots \\
            a_m& (a_m,b_1) & (a_m,b_2) & \cdots & (a_m,b_n)  \\
    \end{array}
  \]
  
\end{definition}

We consider the following setting. Let $\infw{s}=a_0a_1\cdots$ and $\infw{y}=b_0b_1\cdots$ be two infinite words over the alphabets $\mathcal{A}$ and $\mathcal{B}$ respectively. First, build the direct product $\infw{s}\times\infw{y}\in(\mathcal{A}\times\mathcal{B})^{\mathbb{N}\times\mathbb{N}}$ of these two words. Let $f:\mathcal{A}\times \mathcal{B}\to \mathcal{C}$ be a map. For each $b\in \mathcal{B}$, we let $f_b:\mathcal{A}\to \mathcal{C}$ be the map
\[
  f_b(a)=f(a,b).
\]
This coding is extended to a morphism $f_b:\mathcal{A}^*\to \mathcal{C}^*$, i.e., for a word $a_1\cdots a_n$ where $a_i$ is a letter, we have $f_b(a_1\cdots a_n)=f_b(a_1)\cdots f_b(a_n)$. 
We consider the infinite array $\mathbf{M}$ built from $(\mathbf{s},\mathbf{y},f)$ by
\[
  \forall i,j\ge 0,\quad \mathbf{M}[i,j]=f(a_i,b_j)=f_{b_j}(a_i).
\]

\begin{example}
  A simple way to build a two-dimensional morphism (or substitution) is to take the direct product of two unidimensional (prolongable) morphisms $\mu:\mathcal{A}^*\to\mathcal{A}^*$ and $\nu:\mathcal{B}^*\to\mathcal{B}^*$. The resulting morphism is defined on the alphabet $\mathcal{A}\times\mathcal{B}$ and $(a,b)\mapsto \mu(a)\times \nu(b)$, i.e., the image of the pair $(a,b)$ is a rectangular array of size $|\mu(a)|\times |\nu(b)|$ whose element $(i,j)$ is $([\mu(a)]_i,[\nu(b)]_j)$. See, for instance, \cite{Priebe}. Then one may apply an extra coding $f:\mathcal{A}\times\mathcal{B}\to\mathcal{C}$ with the required properties.

For instance, with $\mathbf{s}=\mathbf{y}=01101001\cdots$ being two copies of the Thue--Morse word over $\{0,1\}$, we have depicted on the left of \cref{fig:tm2} the corresponding direct product $\mathbf{s}\times\mathbf{y}$ over a $4$-letter (colored) alphabet. On the right, we have obtained the two-dimensional Thue--Morse array~$\infw{T}$ obtained by coding $(0,0)$ and $(1,1)$ as $0$ (in white) and $(0,1)$ and $(1,0)$ as $1$ (in green). So, in our setting, the coding $f$ is addition of the components modulo~$2$.
  \begin{figure}[h!t]
    \centering
    \includegraphics[width=10cm]{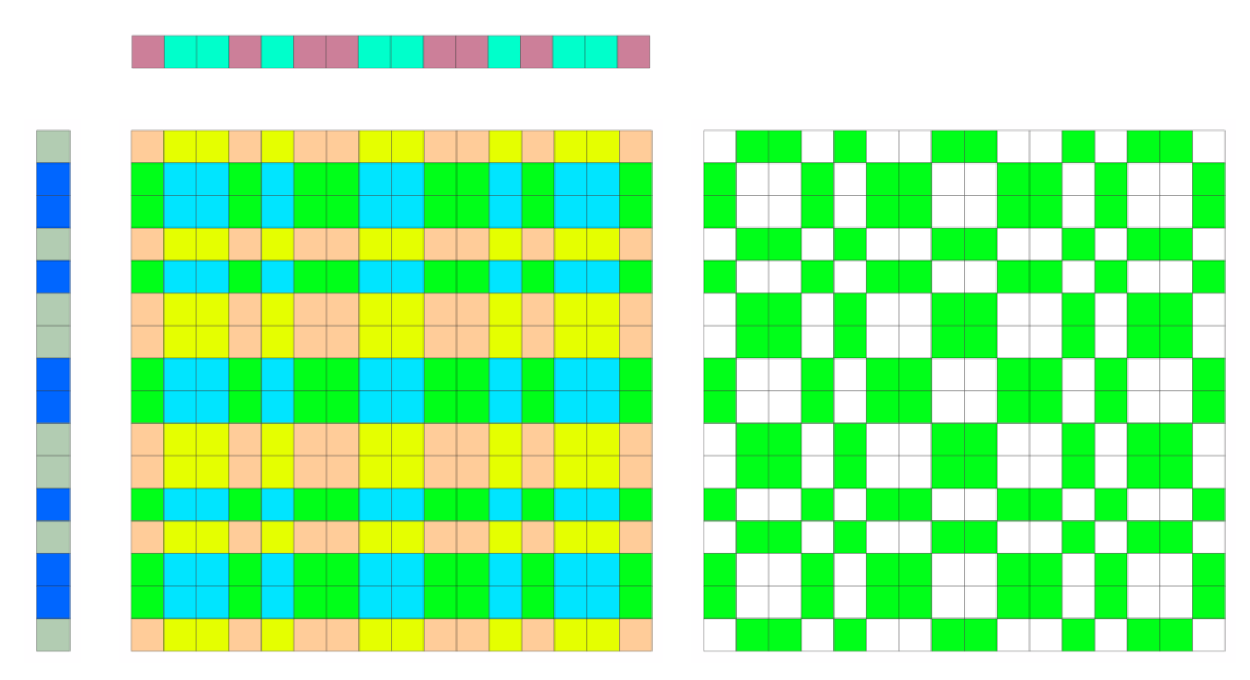}
    \caption{Direct product $\infw{t}\times\infw{t}$ of Thue--Morse words and its coding~$\infw{T}$.}
    \label{fig:tm2}
  \end{figure}
The array~$\infw{T}$ is the fixed point of Thue--Morse morphism~$\Phi$ defined by 
\begin{equation}
  \label{eq:Psi}
  \Phi:0\mapsto
  \begin{array}{|c|c|}
      \hline
    0&1\\
       \hline
    1&0\\
       \hline
  \end{array}\ ,\quad
  1\mapsto
  \begin{array}{|c|c|}
  \hline
    1&0\\
       \hline
    0&1\\
       \hline
  \end{array}\ .
\end{equation}

In \cref{fig:fibtm}, we have considered the direct product of the Fibonacci word~$\infw{f}$ with itself and then, with the Thue--Morse word.
\begin{figure}[h!t]
  \centering
  \includegraphics[width=10cm]{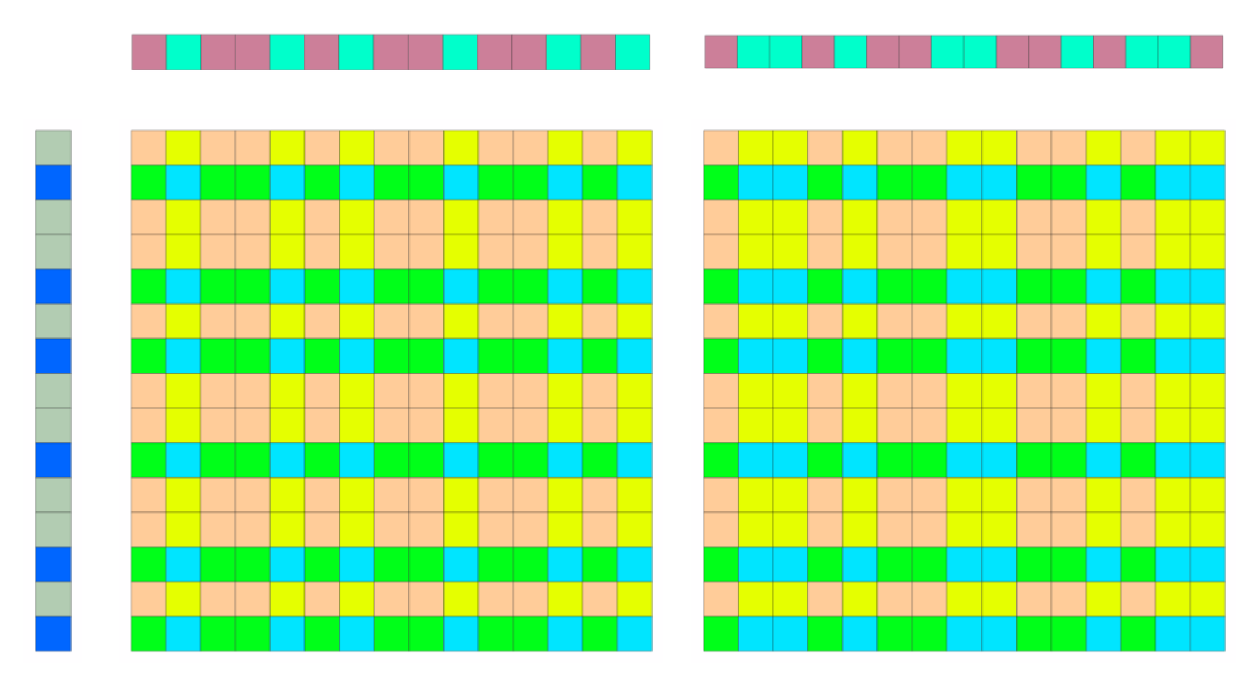}
  \caption{Direct products $\infw{f}\times\infw{f}$ and $\infw{f}\times\infw{t}$ of Fibonacci and Thue--Morse words.}
  \label{fig:fibtm}
\end{figure}
The array $\infw{f}\times\infw{f}$ is fixed by the morphism
\[
  0\mapsto
  \begin{array}{|c|c|}
      \hline
    0&1\\
       \hline
    2&3\\
       \hline
  \end{array}\ ,\quad 
  1\mapsto
  \begin{array}{|c|}
      \hline
    0\\
       \hline
    2\\
       \hline
  \end{array}\ ,\quad
   2\mapsto
  \begin{array}{|c|c|}
      \hline
    0&1\\
       \hline
  \end{array}\ ,\quad 
   3\mapsto
  \begin{array}{|c|c|}
      \hline
    0\\
       \hline
  \end{array}
\]
where $0,1,2,3$ stand resp. for $(0,0)$, $(0,1)$, $(1,0)$ and $(1,1)$. 
This is an example of a shape-symmetric morphism, see \cite{CKR2010}. The $(k,\ell)$-binomial complexity of $\infw{f}\times\infw{f}$ is given in \cref{cor:sturmian}.
\end{example}

\begin{theorem}\label{thm:bounded_complexity}
 Let $k,\ell\ge 1$ be integers. The $(k,\ell)$-binomial complexity of the infinite array $\mathbf{M}$ built from $(\mathbf{s},\mathbf{y},f)$ by, for all $i,j\ge 0$, $\mathbf{M}[i,j]=f(s_i,y_j)$ as defined above is such that
\[
  \bc{\infw{M}}{k,\ell}(m,n)\le\bc{\infw{s}}{k}(m).\bc{\infw{y}}{\ell}(n)\quad \forall m,n\ge 1.
\]
In particular, if  $\mathbf{s}$ and $\mathbf{y}$ are fixed points of Parikh-collinear morphisms, then $\bc{\infw{M}}{k,\ell}$ is bounded by a constant.
\end{theorem}

\begin{proof}
  Without loss of generality, we may assume $m\ge k$, $n\ge \ell$. Indeed, if the result is proved for large enough values then, when $m<k$, we get 
  \[
    \bc{\infw{M}}{k,\ell}(m,n)=\bc{\infw{M}}{m,\ell}(m,n)
    \le
    \bc{\infw{s}}{m}(m).\bc{\infw{y}}{\ell}(n)=
    \bc{\infw{s}}{k}(m).\bc{\infw{y}}{\ell}(n)\]
  because, if $M$ has $m$ rows (and thus $M$ has no subarray with more than $m$ rows) then $\Sp_{\le k\times\ell}(M)=\Sp_{\le m\times\ell}(M)$. The reasoning is similar when $n<\ell$.
 
Let $P$ be an array of size $k\times \ell$, thanks to \cref{pro:limitkl} we limit ourselves to such arrays. Let $i,j\ge 0$ and  we want to bound the number of values that can be taken by the binomial coefficient
\[
  \binom{\mathbf{M}[i\cdots i+m-1,j\cdots j+n-1]}{P}.
\]
In what follows, unidimensional objects, i.e., rows or columns are interpreted as words. We therefore allow to compute the binomial coefficient of a row $r$ with a column $c$. Hence $\binom{r}{c}$ should be interpreted as $\binom{r}{c^\top}$ to be (potentially) non-zero.

First note that, for all $j$, 
\begin{align*}
  \mathbf{M}[i\cdots i+m-1,j]&=f(s_i,y_j) \cdots f(s_{i+m-1},y_j)\\
                      &=f_{y_{j}}(s_i)\cdots f_{y_{j}}(s_{i+m-1})\\
                      &=f_{y_{j}}(s_i\cdots s_{i+m-1})=f_{y_j}(\mathbf{s}[i\cdots i+m-1])
\end{align*}
where we use the fact that $f_{y_j}$ is a morphism and the convention of the column on the l.h.s. is interpreted as a row on the r.h.s. 
By definition, the binomial coefficient is equal to
\begin{align}
   &\sum_{j\le j_1< \cdots <j_\ell \le j+n-1} \binom{\mathbf{M}[i\cdots i+m-1,j_1]}{P_1}\cdots \binom{\mathbf{M}[i\cdots i+m-1,j_\ell]}{P_\ell} \nonumber \\
  =&\sum_{j\le j_1< \cdots <j_\ell \le j+n-1} \binom{f_{y_{j_1}}(\mathbf{s}[i\cdots i+m-1])}{P_1}\cdots \binom{f_{y_{j_\ell}}(\mathbf{s}[i\cdots i+m-1])}{P_\ell} \nonumber \\
  =&\sum_{u_1\cdots u_\ell\in \mathcal{B}^\ell} \binom{\mathbf{y}[j\cdots j+n-1]}{u_1\cdots u_\ell}
     \ \prod_{t=1}^\ell \binom{f_{u_t}(\mathbf{s}[i\cdots i+m-1])}{P_t}  \label{eq:develop} \\
  =&\sum_{u_1\cdots u_\ell\in \mathcal{B}^\ell} \binom{\mathbf{y}[j\cdots j+n-1]}{u_1\cdots u_\ell} 
   \ \prod_{t=1}^\ell \sum_{\substack{w_t\in\mathcal{A}^k\\ f_{u_t}(w_t)=P_t}} \binom{\mathbf{s}[i\cdots i+m-1]}{w_t}  \nonumber
\end{align} 
where, we used the fact that each $f_{u_t}$ is a coding. 

The Parikh vector $\Psi_\ell(\mathbf{y}[j\cdots j+n-1])$ indexed by the words of $\mathcal{B}^\ell$ genealogically ordered 
\[
\left(\binom{\mathbf{y}[j\cdots j+n-1]}{z}\right)_{z\in\mathcal{B}^\ell}
\]
takes exactly $\bc{\infw{y}}{\ell}(n)$ values when $j$ varies. Similarly, the vector indexed by the words $u=u_1\cdots u_\ell$ of $\mathcal{B}^\ell$ whose entries are of the form 
\[
  \prod_{t=1}^\ell \sum_{\substack{w_t\in\mathcal{A}^k\\ f_{u_t}(w_t)=P_t}} \binom{\mathbf{s}[i\cdots i+m-1]}{w_t}
\]
takes at most $\bc{\infw{s}}{k}(m)$ pairwise distinct values when $i$ varies. The above computation shows that the coefficient \eqref{eq:develop} is the dot product of these two vectors. It can therefore take at most $\bc{\infw{s}}{k}(m).\bc{\infw{y}}{\ell}(n)$ distinct values.

Note that the two vectors and thus \eqref{eq:develop} remain unchanged when replacing $\mathbf{y}[j\cdots j+n-1]$ with an $\ell$-binomially equivalent factor and $\mathbf{s}[i\cdots i+m-1]$ with a $k$-binomially equivalent one. 
\qed
\end{proof}

\begin{lemma}
  Let $x_1,\ldots,x_s,y_1,\ldots,y_t\in\mathbb{R}$ and $S_x:=\sum_i x_i$, $S_y:=\sum_i y_i$. If $S_x.S_y\neq 0$, then the knowledge of
   $S_x$, $S_y$ and $A=(x_iy_j)_{1\le i\le s,1\le j\le t}$ uniquely determine $x_1,\ldots,x_s,y_1,\ldots,y_t$.
\end{lemma}

\begin{proof}
  The row and column sums of the matrix $A$ are
  \[
    R_i = \sum_{j=1}^t a_{i,j}= x_i S_y,\quad
    C_j = \sum_{i=1}^s a_{i,j}= y_j S_x.
  \]
  Hence $x_i=R_i/S_y$ and $y_j=C_j/S_x$. \qed
\end{proof}

\begin{proposition}\label{pro:complex_prod}
  Let $k,\ell\ge 1$ be integers and $\infw{s}\in\mathcal{A}^\mathbb{N}$, $\infw{y}\in\mathcal{B}^\mathbb{N}$ be infinite words. The $(k,\ell)$-binomial complexity of the direct product $\infw{s}\times\infw{y}$ satisfies
  \[
    \bc{\infw{s}\times\infw{y}}{k,\ell}(m,n)= \bc{\infw{s}}{k}(m).\bc{\infw{y}}{\ell}(n),\quad \forall m,n\ge 1.
  \]
\end{proposition}

\begin{proof}

  Let $u\in\Fac_m(\infw{s})$ and $v\in\Fac_n(\infw{y})$.   As in the proof of \cref{thm:bounded_complexity}, we may assume without loss of generality that $m\ge k$, $n\ge\ell$.

  Let $\pi_1$ (resp. $\pi_2$ be the morphism projecting a word $(x_1,y_1)\cdots (x_p,y_p)$ over $\mathcal{A}\times\mathcal{B}$ to its first component $x_1\cdots x_p$ (resp. second component $y_1\cdots y_p$).
  
  Let $P$ be an array of size $r\times s$ over $\mathcal{A}\times\mathcal{B}$ with $r\le k$ and $s\le\ell$. Assume first that, for each column $P_j$ of $P$, the second component of each element is constant: i.e., $\pi_2(P_{i,j})=b_j$ for all $i\in\{1,\ldots,r\}$ for some $b_j\in\mathcal{B}$. Then 
  
  \[
    \binom{u\times v}{P}=\binom{v}{b_1\cdots b_s}\binom{u}{\pi_1(P_1)}\cdots \binom{u}{\pi_1(P_s)}.
  \]
We will use this relation with $1\times 1$, $k\times 1$ and $1\times \ell$ arrays. Note that when a column of $P$ has a non-constant second component, then $\binom{u\times v}{P}=0$ because $P$ cannot occur in a direct product.

  Assume that $u\times v\sim_{k,\ell} u'\times v'$, we will prove that $u\sim_k u'$ and $v\sim_\ell v'$. This shows that $\bc{\infw{s}\times\infw{y}}{k,\ell}(m,n)\ge \bc{\infw{s}}{k}(m).\bc{\infw{y}}{\ell}(n)$ which is enough thanks to the previous theorem.

  Let $(a,b)\in\mathcal{A}\times\mathcal{B}$. We have
  \[
    |u|_a.|v|_b=\binom{u\times v}{(a,b)}=  \binom{u'\times v'}{(a,b)}=|u'|_a.|v'|_b
  \]
  and $\sum_{a\in\mathcal{A}} |u|_a=m=\sum_{a\in\mathcal{A}} |u'|_a$, $\sum_{b\in\mathcal{B}} |v|_b=n=\sum_{b\in\mathcal{B}} |v'|_b$. Thanks to the above lemma, $u$ and $u'$ (resp. $v$ and $v'$) have the same Parikh vector.

  Let $b\in\mathcal{B}$ be such that $|v|_b>0$. Let $w\in\mathcal{A}^k$. So, we get
   \[
    |v|_b\binom{u}{w}=\binom{u\times v}{w\times b}=  \binom{u'\times v'}{w\times b}=|v'|_b\binom{u'}{w}
  \]
  and $u\sim_k u'$.

  Finally, let $a\in\mathcal{A}$ be such that $|u|_a>0$. Let $z\in\mathcal{B}^\ell$. We get
  \[
    \binom{v}{z} |u|_a^{|z|} =\binom{u\times v}{a\times z}=  \binom{u'\times v'}{a\times z}= \binom{v'}{z} |u'|_a^{|z|} 
  \]
  and $v\sim_\ell v'$.
  \qed
\end{proof}

\begin{corollary}\label{cor:sturmian}
  Let $\infw{s}$ and $\infw{y}$ be two Sturmian words. For all $k,\ell,m,n\ge 1$, we have  
  \[
    \bc{\infw{s}\times\infw{y}}{k,\ell}(m,n)=\left\{\begin{array}{ll}
      4,&\text{ if }k=\ell=1;\\
      2(n+1),&\text{ if }k=1, \ell\ge 2;\\
      2(m+1),&\text{ if }k\ge 2, \ell=1;\\
      (m+1)(n+1),&\text{ if }k,\ell\ge 2.\\
    \end{array}\right.
\]
\end{corollary}

\section{Thue--Morse word, bit-wise complement and binomial equivalence}\label{sec:bitwise}

Before considering a two-dimensional setting, we collect some useful facts about binomial equivalence classes of the Thue--Morse word. We let $\overline{u}$ denote the image of the word $u\in\{0,1\}^*$ under the involution morphism $i\mapsto 1-i$, also called {\em bit-wise complement}. Our aim is to describe precisely when an equivalence class $[u]_k$ is fixed by bit-wise complement.

The reader may already notice that for $n\ge 2^k$, from \cref{thm:lejeune}, $\bc{\infw{t}}{k}(n)$ is odd if and only if $n\equiv 0\pmod{2^k}$. Our aim in this section is to show that this is the only case where there is a class such that $[u]_k=[\overline{u}]_k$. In all other cases, the equivalence classes occur in pairs $[u]_k\neq [\overline{u}]_k$.

\begin{remark}
  Note that, for all words $u,w\in\{0,1\}^*$, we have $\binom{\overline{u}}{w}=\binom{u}{\overline{w}}$. Hence
  \begin{equation}\label{eq:bitwise}
    \binom{u}{w}=\binom{\overline{u}}{w} \Leftrightarrow \binom{u}{w}=\binom{u}{\overline{w}}.
  \end{equation}
\end{remark}

Let $m\ge 1$.  It is quite usual, when one aims to distinguish two words $x$ and $y$ that are not $k$-binomially equivalent, to look for a specific subword $z$ of length~$k$ such that the corresponding binomial coefficients $\binom{x}{z}$ and $\binom{y}{z}$ are distinct. Here we consider the two particular words $\alpha_m$ and $\beta_m=\overline{\alpha_m}$ where $\alpha_m$ is the word of length $m$ starting with $0$ and alternating $0$ and $1$. For instance, $\alpha_4=0101$ and $\beta_4=1010$. We also set $\alpha_0=\beta_0=\varepsilon$. Note that this particular alternating word was already considered in \cite{GRW}. 
  As a consequence of Chu--Vandermonde identity, we have 
  \[
    \binom{uv}{\alpha_m}=\sum_{\substack{i\text{ even}\\ 0\le i\le m}}
    \binom{u}{\alpha_i} \binom{v}{\alpha_{m-i}}+
      \sum_{\substack{i\text{ odd}\\ 0\le i\le m}}
      \binom{u}{\alpha_i} \binom{v}{\beta_{m-i}}
    \]
      and
  \[
    \binom{uv}{\beta_m}=\sum_{\substack{i\text{ even}\\ 0\le i\le m}}
    \binom{u}{\beta_i} \binom{v}{\beta_{m-i}}+
      \sum_{\substack{i\text{ odd}\\ 0\le i\le m}}
      \binom{u}{\beta_i} \binom{v}{\alpha_{m-i}}.
    \]
    
  \begin{definition}\label{def:deltasigma}
  Let $m\ge 0$. Set
  \[
    \delta_m(u):=\binom{u}{\alpha_m}-\binom{u}{\beta_m}=\binom{u}{\alpha_m}-\binom{\overline{u}}{\alpha_m}=\binom{\overline{u}}{\beta_m}-\binom{u}{\beta_m}
  \]
  and
   \[
    \sigma_m(u):=\binom{u}{\alpha_m}+\binom{u}{\beta_m}.
  \]
  \end{definition}
We will often make use of the fact that
\begin{equation}\label{eq:simplify_sigma}
  \sigma_0(u)=2 \text{ and }\sigma_1(u)=|u|.
\end{equation}

  Now, since $\alpha_m$ and $\beta_m$ alternate $0$'s and $1$'s, Chu--Vandermonde identity gives
  \begin{align*}
    \delta_m(uv)&=\sum_{\substack{i\text{ even}\\ 0\le i\le m}}
    \left[ \binom{u}{\alpha_i} \binom{v}{\alpha_{m-i}}-
    \binom{u}{\beta_i} \binom{v}{\beta_{m-i}} \right] \\
    & +
      \sum_{\substack{i\text{ odd}\\ 0\le i\le m}}
    \left[   \binom{u}{\alpha_i} \binom{v}{\beta_{m-i}}-\binom{u}{\beta_i} \binom{v}{\alpha_{m-i}} \right]
\end{align*}
and use the identity $ab-cd=\frac12[(a-c)(b+d)+(a+c)(b-d)]$ to obtain a convolution of $\delta_i(u)$'s and $\sigma_j(v)$'s
\begin{align}
    \delta_m(uv)&=\frac12 \sum_{\substack{i\text{ even}\\ 0\le i\le m}}
    \left[ \delta_i(u) \sigma_{m-i}(v)+ \sigma_i(u) \delta_{m-i}(v)\right] \nonumber \\
    & +
      \frac12\sum_{\substack{i\text{ odd}\\ 0\le i\le m}}
  \left[ \delta_i(u) \sigma_{m-i}(v)- \sigma_i(u) \delta_{m-i}(v) \right] \nonumber \\
                &= \frac12 \sum_{i=0}^m \delta_i(u) \sigma_{m-i}(v) 
  + \frac12 \sum_{i=0}^m (-1)^i \sigma_i(u) \delta_{m-i}(v). \label{eq:identity_deltam}
\end{align}

  \begin{remark}
 Let $u$ be a word in $\{0,1\}^*$. If the class $[u]_k$ is fixed by bit-wise complement, i.e., $[u]_k=[\overline{u}]_k$, then 
  \begin{equation}\label{eq:class_fixed}
    \delta_m(u)=0,\quad \forall m\le k.
  \end{equation}
  Hence, if there exists $m\le k$ such that $\delta_m(u)\neq 0$, then $[u]_k\neq [\overline{u}]_k$.
  \end{remark}

  By \eqref{eq:Oschenschlager}, $\varphi^k(a)\sim_k\varphi^k(\overline{a})$, $a\in\{0,1\}$, hence we have
  \[
    \delta_m(\varphi^k(a))=0,\quad \forall m\le k
  \]
  and, by the congruence property, the class $[\varphi^k(z)]_k$ is fixed by bit-wise complement, for all $z\in\{0,1\}^*$.

 Recall that $\nu_2(n)$ is the {\em $2$-adic valuation} of $n$, the exponent of the highest power of $2$ dividing the non-zero integer $n$.   
  \begin{lemma}\label{lem:delta_prefix}
    Let $k\ge 1$. For any non-empty proper prefix (resp. suffix) $p$ of $\varphi^k(a)$, $a\in\{0,1\}$, there exists $m\le k$ such that $\delta_m(p)\neq 0$. Moreover, the least such $m$ is equal to $1+\nu_2(|p|)$. In particular, $\delta_k(p)\neq0$ and $\delta_m(p)=0$, for all $m<k$, if and only if $|p|=2^{k-1}$.
  \end{lemma}

  \cref{tab:delta} contains the value of the first few $\delta$-functions for proper prefixes of $\varphi^4(0)$. In bold, we have highlighted the first level where the value is non-zero.
  \begin{table}[h!t]
  \[
\begin{array}{c|rrrrrrrrrrrrrrr}
|p| & 1 & 2 & 3 & 4 & 5 & 6 & 7 & 8 & 9 & 10 & 11 & 12 & 13 & 14 & 15 \\
  \hline
  \nu_2 & 0&1&0&2&0&1&0&3&0&1&0&2&0&1&0\\
\delta_1 & \mathbf{1} & 0 & \mathbf{-1} & 0 & \mathbf{-1} & 0 & \mathbf{1} & 0 & \mathbf{-1} & 0 & \mathbf{1} & 0 & \mathbf{1} & 0 & \mathbf{-1} \\
\delta_2 & 0 & \mathbf{1} & 2 & 0 & 2 & \mathbf{-1} & -4 & 0 & 4 & \mathbf{-1} & -6 & 0 & -6 & \mathbf{1} & 8 \\
\delta_3 & 0 & 0 & 0 & \mathbf{2} & 0 & 4 & 8 & 0 & -8 & 4 & 16 & \mathbf{-2} & 16 & -8 & -32 \\
\delta_4 & 0 & 0 & 0 & 0 & 2 & 0 & -2 & \mathbf{8} & 18 & 0 & -18 & 16 & -20 & 32 & 84 \\
\end{array}
\]
\caption{The first few $\delta$-functions for prefixes of $\varphi^4(0)$.}\label{tab:delta}
  \end{table}
  
  \begin{proof}
    We prove the result for prefixes by induction on $k\ge 1$. The proof is the same for suffixes. W.l.o.g. take $a=0$. If $k=1$, $p=0$ and $\delta_1(0)=1$.

    Assume that the property holds for $k\ge 1$ and prove it for $k+1$. The word~$p$ is a proper prefix of $\varphi^{k+1}(0)=\varphi^k(0)\varphi^k(1)$. If $|p|<2^k$, then we may apply the induction hypothesis.

    Now assume that $p=\varphi^k(0)r$.

    If $r$ is empty, since $\varphi^k(0)\sim_k\varphi^k(1)$, we get $\delta_m(\varphi^k(0))=0$ for all $m\le k$. We know from \cite[Prop.~4.1]{GRW} that 
\begin{equation}\label{eq:deltak1}
  \delta_{k+1}(\varphi^k(0))=\binom{\varphi^k(0)}{\alpha_{k+1}}-\binom{\varphi^k(0)}{\beta_{k+1}}=2^{\binom{k}{2}}\neq 0.
\end{equation}
Finally assume that $r$ is a non-empty proper prefix of $\varphi^k(1)$. Let $R:=1+\nu_2(|r|)\le k$. By induction hypothesis, $\delta_{R}(r)\neq 0$ and  $\delta_{n}(r)= 0$ for $n<R$. Using the fact that $\varphi^k(0)\sim_k \varphi^k(1)$ and $R\le k$, we get
 \begin{align*}
   \delta_R(\varphi^k(0)r)&=\binom{\varphi^k(0)r}{\alpha_R}-\binom{\varphi^k(0)r}{\beta_R}\\
                          &= \sum_{\substack{i\text{ even}\\ 0\le i\le R}}
    \binom{\varphi^k(0)}{\alpha_i} \binom{r}{\alpha_{R-i}}+
      \sum_{\substack{i\text{ odd}\\ 0\le i\le R}}
   \binom{\varphi^k(0)}{\alpha_i} \binom{r}{\beta_{R-i}}\\
   &-\sum_{\substack{i\text{ even}\\ 0\le i\le R}}
    \binom{\varphi^k(0)}{\beta_i} \binom{r}{\beta_{R-i}}-
      \sum_{\substack{i\text{ odd}\\ 0\le i\le R}}
   \binom{\varphi^k(0)}{\beta_i} \binom{r}{\alpha_{R-i}}\\
   &= \sum_{\substack{i\text{ even}\\ 0\le i\le R}}
    \binom{\varphi^k(0)}{\alpha_i} \delta_{R-i}(r)-
      \sum_{\substack{i\text{ odd}\\ 0\le i\le R}}
   \binom{\varphi^k(0)}{\alpha_i} \delta_{R-i}(r)\\
   &= \sum_{i=0}^R (-1)^i 
    \binom{\varphi^k(0)}{\alpha_i} \delta_{R-i}(r)=\delta_R(r).
 \end{align*}
 The same calculation gives $\delta_j(\varphi^k(0)r)=0$ for all $j<R$.
 Since $|\varphi^k(0)|=2^k$ and $|r|<2^k$, $\nu_2(2^k+|r|)=\nu_2(|r|)$.
 \qed
  \end{proof}

  \begin{lemma}\label{lem:new5}
    Let $a\in\{0,1\}$ and $v\in\{0,1\}^*$. For all $\ell\ge 1$,
    \[
      \delta_{\ell+1}\left( \varphi^{\ell-1}(a)\varphi^\ell(v)\varphi^{\ell-1}(\overline{a}) \right) = 2^{\binom{\ell}{2}}\left( \delta_1(v) + (-1)^{a} (2|v|+1)\right).
    \]
    In particular, $\delta_{\ell+1}\left( \varphi^{\ell-1}(a)\varphi^\ell(v)\varphi^{\ell-1}(\overline{a}) \right)\neq 0$.
  \end{lemma}

  \begin{proof}
     Recall from \eqref{eq:class_fixed} that, for all $i\le \ell-1$,
    \[
      \delta_i(\varphi^{\ell-1}(a))=\delta_i(\varphi^{\ell-1}(\overline{a}))=0
    \]
    and, for all $j\le\ell$, $\delta_j(\varphi^{\ell}(v))=0$.
Moreover, \eqref{eq:deltak1} (or $\delta_1(a)=(-1)^{a}$ when $\ell=1$) gives
    \[
      \delta_\ell(\varphi^{\ell-1}(a))=(-1)^{a} 2^{\binom{\ell-1}{2}}=-\delta_\ell(\varphi^{\ell-1}(\overline{a})).
    \]
    
    Note that, \eqref{eq:deltak1} can be extended to a word $v=v_1\cdots v_t$ as follows
    \[
      \delta_{\ell+1}(\varphi^\ell(v))=2^{\binom{\ell}{2}} \delta_1(v).
    \]
    Indeed,
    \[
      \delta_{\ell+1}(\varphi^\ell(v))=\binom{\varphi^\ell(v_1)\cdots \varphi^\ell(v_t)}{\alpha_{\ell+1}}-\binom{\varphi^\ell(v_1)\cdots\varphi^\ell(v_t)}{\beta_{\ell+1}}.
    \]
    Whenever an occurrence of \(\alpha_{\ell+1}\) or \(\beta_{\ell+1}\) uses at least two \(\ell\)-blocks, at most \(\ell\) letters are selected from each block. Those contributions cancel because $\varphi^\ell(0)\sim_\ell\varphi^\ell(1)$. Only occurrences contained entirely in one block remain. By \eqref{eq:deltak1}, a \(0\)-block contributes \(2^{\binom{\ell}{2}}\) and a \(1\)-block contributes negatively.

    We may now apply the convolution formula \eqref{eq:identity_deltam} twice and \eqref{eq:simplify_sigma} to get
    \begin{align*}
      \delta_{\ell+1}\left( \varphi^{\ell-1}(a)\varphi^\ell(v)\varphi^{\ell-1}(\overline{a}) \right) &=& \delta_{\ell+1}(\varphi^\ell(v))+\delta_\ell(\varphi^{\ell-1}(a))\left( 2^\ell |v|+2^{\ell-1}\right)\\
      &=& 2^{\binom{\ell}{2}} \delta_1(v) + (-1)^{a} 2^{\binom{\ell-1}{2}} 2^{\ell-1} \left( 2|v|+1\right).
    \end{align*}
    The expression is non-zero because $|\delta_1(v)|\le |v|<2|v|+1$.
    \qed
  \end{proof}
  
\begin{proposition}\label{pro:counting_classes}
  Let $k\ge 1$. Consider the quotient of the set of factors of length~$n\ge 2^k$ of the Thue--Morse word by the $k$-binomial equivalence relation. The bit-wise complement maps every word of a $k$-binomial equivalence class to a word of another class and conversely, or it fixes a class. If $2^k|n$, then exactly one $k$-binomial equivalence class is fixed by bit-wise complement. Otherwise, no equivalence class is fixed.
\end{proposition}

\begin{proof}
  Let $u,v\in\{0,1\}^*$. It is straightforward to see that $u\sim_k v$ if and only if $\overline{u}\sim_k\overline v$. Thus, the bit-wise complement maps every word of a $k$-binomial equivalence class~$[u]_k$ to a word of the class~$[\overline{u}]_k$ and conversely, or it fixes a class. Hence bit-wise complement permutes the $k$-binomial classes.
 
 By recognizability of the $2$-uniform Thue--Morse morphism~$\varphi$, any factor~$u$ of length $n\ge 2^k$ occurring in $\infw{t}$ is of the form
 \[
   u=s\, \varphi^k(z)\, p
 \]
 where $s$ (resp. $p$) is a proper suffix (resp. prefix) of the image of a letter by $\varphi^k$, and $z$ may be empty.

 Let us show that  $[u]_k$ is fixed by bit-wise complement if and only if $s=p=\varepsilon$ or, $s=\varphi^{k-1}(a)$ and $p=\varphi^{k-1}(\overline{a})$ with $a\in\{0,1\}$. Note that, in that case $|u|=|ps|+|z|2^k$ is divisible by $2^k$.

The condition is sufficient. Assume $s=p=\varepsilon$. Since $\sim_k$ is a congruence, by Ochsenschl\"ager's result \eqref{eq:Oschenschlager}, we know that, for all words $z\in\{0,1\}^+$,  
   \[
     \varphi^k(z)\sim_k \varphi^k(\overline{z})
   \]
   and then the class $[\varphi^k(z)]_k$ is fixed by complement.

Assume $s=\varphi^{k-1}(a)$ and $p=\varphi^{k-1}(\overline{a})$ with $a\in\{0,1\}$. We use the transfer lemma,
   \begin{align*}
     \varphi^{k-1}(a)\varphi^k(z)\varphi^{k-1}(\overline{a})& \sim_k
     \varphi^{k-1}(a)\varphi^{k-1}(\overline{a})\varphi^k(z)\sim_k
                                                              \varphi^k(az)\sim_k\varphi^k(\overline{az})\\
     & \sim_k
     \varphi^{k-1}(\overline{a})\varphi^{k-1}(a)\varphi^k(\overline{z})\sim_k
     \varphi^{k-1}(\overline{a})\varphi^k(\overline{z})\varphi^{k-1}(a)
   \end{align*}
   and then the class $[\varphi^{k-1}(a)\varphi^k(z)\varphi^{k-1}(\overline{a})]_k$ is fixed by complement. Note that this class contains any word of the form $\varphi^k(w)$ where $w$ is a factor of $\infw{t}$ of length $|z|+1$. This is thus the unique class fixed by bit-wise complement.

The condition is necessary. Assume now that $[u]_k=[\overline{u}]_k$ and $|sp|>0$. We have to show that this implies $s=\varphi^{k-1}(a)$ and $p=\varphi^{k-1}(\overline{a})$ with $a\in\{0,1\}$.

\begin{itemize}
\item[a)]
  If only one of $p$ or $s$ is non-empty --- w.l.o.g. assume $p\neq\varepsilon$ --- let $R=1+\nu_2(|p|)$. A computation similar to the one from the proof of \cref{lem:delta_prefix} shows that
\[
  \delta_R(u)=\delta_R(p)\neq 0
\]
which, by \eqref{eq:class_fixed}, means that $[u]_k\neq [\overline{u}]_k$. A contradiction.

\item[b)] Now assume that both $s$ and $p$ are non-empty and set $\ell:=1+\min\{\nu_2(|s|),\nu_2(|p|)\}$. In particular, at least one of the two quantities $\delta_\ell(s)$ or $\delta_\ell(p)$ is non-zero and $\delta_i(s)=\delta_i(p)=0$ for all $i<\ell$. 

  We have, applying the convolution formula \eqref{eq:identity_deltam}, 
  \begin{equation}\label{eq:delta_ell}
\delta_\ell(s\varphi^k(z))= \frac12 \sum_{i=0}^\ell \delta_i(s) \sigma_{\ell-i}(\varphi^k(z)) 
  + \frac12 \sum_{i=0}^\ell (-1)^i \sigma_i(s) \delta_{\ell-i}(\varphi^k(z))
\end{equation}
and by definition of $\ell$, $\delta_i(s)=0$ for $i<\ell$. Note that $\sigma_0(w)=2$, for any word $w$, and $\delta_i(\varphi^k(z))=0$ for all $i\le k$ because $\varphi^k(z)\sim_k
\varphi^k(\overline{z})$. Hence, we conclude that
 \[
   \delta_\ell(s\varphi^k(z))=\delta_\ell(s).
 \]
 Moreover $\delta_i(s\varphi^k(z))=0$ for $i<\ell$.
We are now in a position to compute
   \[
\delta_\ell(s\varphi^k(z)p)= \frac12 \sum_{i=0}^\ell \delta_i(s\varphi^k(z)) \sigma_{\ell-i}(p) 
  + \frac12 \sum_{i=0}^\ell (-1)^i \sigma_i(s\varphi^k(z)) \delta_{\ell-i}(p).
\]
The first term is equal to $\delta_\ell(s)$ thanks to the above discussion. By definition of $\ell$, $\delta_i(p)=0$ for $i<\ell$. Hence the second term is $\delta_\ell(p)$ and
\[
  \delta_\ell(u)=\delta_\ell(s)+\delta_\ell(p).
\]
Since $[u]_k=[\overline{u}]_k$, we have $\delta_\ell(u)=0$ hence $\delta_\ell(s)=-\delta_\ell(p)\neq 0$. Since $\delta_i(s)=\delta_i(p)=0$ for all $i<\ell$ and $\delta_\ell(s)=-\delta_\ell(p)\neq 0$, \cref{lem:delta_prefix} implies
\begin{equation}\label{eq:13}
\nu_2(|s|)=\nu_2(|p|)=\ell-1.
\end{equation}
We claim that \(\ell=k\).
Assume to the contrary that \(\ell<k\). Since \(2^{\ell-1}\) divides both
\(|s|\) and \(|p|\), and \(s\) (resp. \(p\)) is a suffix (resp. prefix) of a
\(\varphi^k\)-block, there exist a suffix \(s'\) and a prefix \(p'\) of
\(\varphi^{k-\ell+1}\)-blocks such that
\[
s=\varphi^{\ell-1}(s')
\quad\text{ and }\quad
p=\varphi^{\ell-1}(p').
\]Moreover, by \eqref{eq:13}, both \(|s'|\) and \(|p'|\) are odd. Thus
\[
u=\varphi^{\ell-1}(u')
\quad\text{where}\quad
u'=s'\varphi^{k-\ell+1}(z)p'.
\]Since \(u\sim_k\bar u\), in particular \(u\sim_\ell\bar u\). If
\(\ell>1\), Proposition~\ref{pro:equivTM}, applied with exponent $\ell-1$, therefore gives
\(
u'\sim_1\bar u'
\); for \(\ell=1\) the same conclusion is immediate.
Since \(s'\) is an odd-length suffix of a \(\varphi\)-image and \(p'\) is an
odd-length prefix of a \(\varphi\)-image, there exist
\(a,b\in\{0,1\}\) and words \(x,y\) such that
\[
s'=a\varphi(x)
\quad\text{ and }\quad
p'=\varphi(y)b.
\]
Consequently
\[
u'
=
a\varphi(v)b
\quad\text{where}\quad
v=x\varphi^{k-\ell}(z)y.
\]
As \(\varphi(v)\) contains equally many \(0\)'s and \(1\)'s and
\(u'\sim_1\bar u'\), we must have \(b=\bar a\). Therefore
\[
u=
\varphi^{\ell-1}(a)
\varphi^\ell(v)
\varphi^{\ell-1}(\bar a).
\]
Lemma~\ref{lem:new5} now gives
\[
\delta_{\ell+1}(u)\neq0.
\]
But \(\ell<k\), so \(\ell+1\le k\), contradicting \eqref{eq:class_fixed}. Hence $k=\ell$.

Now, we can make use of Lemma~\ref{lem:delta_prefix}: \(
\nu_2(|s|)=\nu_2(|p|)=k-1
\). Since
\(
0<|s|,|p|<2^k,
\)
this forces
\(|s|=|p|=2^{k-1}\). Otherwise stated, \(
s=\varphi^{k-1}(a)\) and \(p=\varphi^{k-1}(b)\)
for some \(a,b\in\{0,1\}\). It remains only to show \(b=\overline{a}\). Assume towards a contradiction that \(b=a\), then
\[
u
=
\varphi^{k-1}\!\left(a\varphi(z)a\right).
\]
But \(\varphi(z)\) is balanced, so
\(
a\varphi(z)a\not\sim_1
\bar a\,\varphi(\bar z)\,\bar a
\).
For \(k\ge2\), Proposition~\ref{pro:equivTM} gives
\(
u\not\sim_k\bar u
\), a contradiction. For \(k=1\), the same conclusion follows directly from abelian equivalence. Hence \(b=\bar a\) and the conclusion follows.


\end{itemize}
If $n=q2^k$, any factor $w$ of $\infw{t}$ of length $q$ gives a factor
$\varphi^k(w)$ of length $n$ whose class is fixed. Together with the above
necessity argument and the coincidence of the classes in the two sufficient
cases, this proves the stated existence and uniqueness. \qed
\end{proof}

\begin{remark}\label{rem:small_m}
  From \cref{thm:lejeune}, we know that no two distinct factors of length~$n<2^k$ are $\sim_k$-equivalent. In particular, all classes are singletons and $u\not\sim_k\overline{u}$ if $0<|u|<2^k$. 
\end{remark}

\section{The case of the two-dimensional Thue--Morse word}

We start with some immediate observations.
\subsection{Some special situations}
If we consider two arrays whose columns are all $k$-binomially equivalent, then these arrays are $(k,\ell)$-equivalent.
\begin{lemma}\label{lem:kkequiv}
  Let $k,\ell\ge 1$. 
  Let $M,N\in \mathcal{A}^{r\times m}$ be arrays such that, for all $i,j\in\{1,\ldots,m\}$, $M_i\sim_k N_j$. Then $M\sim_{k,\ell}N$.
\end{lemma}

\begin{proof}
  Replacing $k,\ell$ by $\min(k,r),\min(\ell,m)$ does not change the array equivalence to be proved, so we may assume $k\le r$ and $\ell\le m$.
  Let $P\in \mathcal{A}^{k\times \ell}$. Since $M_1\sim_k M_j$, for all $j$, and $P_i$ is a column of length~$k$, we get
  \[\binom{M_1}{P_i}=\binom{M_j}{P_i}
    \quad \forall i\in\{1,\ldots,\ell\}, \forall j\in\{1,\ldots,m\}.
  \]
  Hence, we obtain
  \[
    \binom{M}{P}=\sum_{1\le j_1< \cdots <j_\ell\le m}\binom{M_{1}}{P_1}\cdots\binom{M_{1}}{P_\ell}=\binom{m}{\ell}\binom{M_{1}}{P_1}\cdots\binom{M_{1}}{P_\ell}.
  \]
  Since $M_1\sim_k N_1$, we deduce that $\binom{M}{P}=\binom{N}{P}$ which is enough. \qed
\end{proof}

We can generalize a bit the above observation with the assumption that corresponding columns are pairwise $k$-binomially equivalent.
\begin{lemma}\label{lem:colum-k-equiv}
  Let $k,\ell\ge 1$. 
  Let $M,N\in \mathcal{A}^{r\times m}$ be arrays such that corresponding columns are pairwise $k$-binomially equivalent, i.e., for all $i$, $M_i\sim_kN_i$. Then $M\sim_{k,\ell}N$.
\end{lemma}

\begin{proof}
  As in the previous proof, we may assume $k\le r$ and $\ell\le m$.
  For $P\in\mathcal{A}^{k\times\ell}$, the result is a direct consequence of the definition of the two-dimensional coefficient
   \[
     \sum_{1\le j_1< \cdots <j_\ell\le m}\binom{M_{j_1}}{P_1}\cdots\binom{M_{j_\ell}}{P_\ell}=
     \sum_{1\le j_1< \cdots <j_\ell\le m}\binom{N_{j_1}}{P_1}\cdots\binom{N_{j_\ell}}{P_\ell}
   \]
   because the $P_i$'s are columns of length~$k$. 
  \qed
\end{proof}

\begin{corollary}
  Let $k,\ell\ge 1$. 
  Let $M,N\in \mathcal{A}^{r\times m}$ be arrays such that corresponding columns are pairwise $k$-binomially equivalent, i.e., for all $i$, $M_i\sim_kN_i$. Let $P\in \mathcal{A}^{s\times m}$ and $Q\in \mathcal{A}^{t\times m}$. Then, 
  $P\ominus M \ominus Q\sim_{k,\ell} P\ominus N \ominus Q$.
\end{corollary}

\begin{proof}
  Since $\sim_k$ is a congruence, $M_i\sim_kN_i$ implies that $P_i\ominus M_i \ominus Q_i\sim_{k} P_i\ominus N_i \ominus Q_i$. Hence we may apply \cref{lem:colum-k-equiv}. \qed
\end{proof}

We can also consider horizontal concatenation and a similar result holds.

\begin{lemma}
  Let $k,\ell\ge 1$. 
  Let $M,N\in \mathcal{A}^{r\times m}$ be arrays such that corresponding columns are pairwise $k$-binomially equivalent, i.e., for all $i$, $M_i\sim_kN_i$. Let $P\in \mathcal{A}^{r\times s}$ and $Q\in \mathcal{A}^{r\times t}$. Then, 
  $P\obar M \obar Q\sim_{k,\ell} P\obar N \obar Q$.
\end{lemma}

\begin{proof}
  This is a consequence of \cref{pro:vandermonde}.
\end{proof}

Ochsenschl\"ager's result extends to a two-dimensional setting.
\begin{corollary}
  Let $\Phi$ be the two-dimensional Thue--Morse morphism defined in \eqref{eq:Psi}. For all $k,\ell\ge 1$, 
  we have $\Phi^k(0)\sim_{k,\ell}\Phi^k(1)$.
\end{corollary}
\begin{proof}
  Each column of $\Phi^k(a)$, $a\in\{0,1\}$, is either $\varphi^k(0)$ or $\varphi^k(1)$. The result then follows from \cref{lem:kkequiv} and \eqref{eq:Oschenschlager}. \qed
\end{proof}

\begin{example}
  As an example, the following factors are $\sim_{2,2}$-equivalent:
  \[
    \Phi^2(0)=\begin{pmatrix}
      0&1&1&0\\
      1&0&0&1\\
      1&0&0&1\\
      0&1&1&0\\
    \end{pmatrix}\text{ and }
        \Phi^2(1)=\begin{pmatrix}
      1&0&0&1\\
      0&1&1&0\\
      0&1&1&0\\
      1&0&0&1\\
    \end{pmatrix}.\]
\end{example}

\subsection{Binomial complexity of $2$D Thue--Morse word}

It is convenient to define the direct product of two words followed by a coding which is addition modulo~$2$.

\begin{definition}
  Let $u,v$ be two finite words over~$\{0,1\}$, we let $W=u\oplus v$ denote the array where $W[i,j]=u_i+v_j\bmod{2}$. Otherwise stated, we apply the coding $(a,b)\mapsto a+b\bmod{2}$ to $u\times v$.
\end{definition}

If $\mathsf{p}_\infw{t}(n)=\#\Fac_n(\infw{t})$ is the factor complexity of the Thue--Morse word, then
\[
  \#\Fac_{n\times n}(\infw{T})=\frac{\left(\mathsf{p}_\infw{t}(n)\right)^2}{2}=2, 8, 18, 50, 72, 128, 200, 242, 288, \ldots
\]
and more generally, we have the following.
\begin{proposition} The factor complexity of the two-dimensional Thue--Morse word is given by
\[
  \#\Fac_{m\times n}(\infw{T})=\frac{\mathsf{p}_\infw{t}(m).\mathsf{p}_\infw{t}(n)}{2}.
\]  
\end{proposition}
  This result is rather elementary and can be obtained in several ways (in particular by an automatic approach using {\tt Walnut}). However, the proof is of interest for the counting arguments that we will develop later.
\begin{proof}
Let $n\ge 1$. Note that any $m\times n$ factor $W=\infw{T}[i\cdots i+m-1,j\cdots j+n-1]$ is obtained from two factors $u=\infw{t}[i\cdots i+m-1]$ and $v=\infw{t}[j\cdots j+n-1]$ by $W=u\oplus v$. 

The map
\[
  \Fac_{m}(\infw{t})\times\Fac_{n}(\infw{t})\to\Fac_{m\times n}(\infw{T}), (u,v)\mapsto u\oplus v
\]
is surjective.

The set of factors of $\infw{t}$ is closed under the bit-wise complement mapping a word $u$ to $\overline{u}$. Observe that $u\oplus v=u'\oplus v'$ if and only if $(u',v')=(u,v)$ or $(u',v')=(\overline{u},\overline{v})$. The conclusion follows.

Equivalently, $W$ and the first letter of $v$ are enough to recover $u$ and $v$. Indeed, the first column of $W$ and $v_1$ determine $u$. Then, $u$ and the $j^{\text{th}}$ column of $W$ determine $v_j$. Starting the process with $\overline{v_1}$ leads to $\overline{u}$ and $\overline{v}$. 
\qed
\end{proof}

It is always tempting to run some computer experiments. Computing the first few values of the $2$ or $3$-binomial complexity leads  to conjecture the following sequences:

\[
  \left( \bc{\infw{T}}{2,2}(n,n)\right)_{n\ge 1}=2, 8, 18, (37, 32, 32, 32)^\omega
\]
and
\[
  \left( \bc{\infw{T}}{3,3}(n,n)\right)_{n\ge 1}= 2, 8, 18, 50, 72, 128, 200, (211, 200, 200, 200, 200, 200, 200, 200)^\omega.
\]
Our aim is to obtain the properties of such sequences.

\begin{lemma}\label{lem:hamming}
Let $k\ge 1$ and $z_1<z_2<\cdots <z_{2^k}$ be the words of $\{0,1\}^k$ genealogically ordered. Consider the matrix
$$M=
\begin{pmatrix}
  \alpha & \beta \\
  \beta & \alpha \\
\end{pmatrix}$$
The $k$-fold Kronecker product $M^{\otimes k}$ is such that
\[
  [M^{\otimes k}]_{i,j}=\alpha^{k-h(z_i,z_j)} \beta^{h(z_i,z_j)}
  \]
  where $h$ is the Hamming distance.
\end{lemma}  
  \begin{proof}
    Proceed by induction on $k$. \qed
  \end{proof}
  As an example, here is the matrix of Hamming distances for words of length~$2$ and the corresponding matrix $M^{\otimes 2}$
  \[
    \begin{pmatrix}
      0&1&1&2\\
      1&0&2&1\\
      1&2&0&1\\
      2&1&1&0\\
    \end{pmatrix}\quad \text{ and }\quad 
    M^{\otimes 2}=\left(
\begin{array}{cccc}
 \alpha^2 & \alpha \beta & \alpha \beta & \beta^2 \\
 \alpha \beta & \alpha^2 & \beta^2 & \alpha \beta \\
 \alpha \beta & \beta^2 & \alpha^2 & \alpha \beta \\
 \beta^2 & \alpha \beta & \alpha \beta & \alpha^2 \\
\end{array}
\right).
\]
  
We are in a position to state the main result of this section.
\begin{theorem}\label{thm:mainTM}
  Let $k\ge 1$, $\ell\ge 2$ and $m,n\ge 1$, we have
\[
  \bc{\infw{T}}{k,\ell}(m,n)=\left\{\begin{array}{ll}
   \frac{\bc{\infw{t}}{k}(m).\bc{\infw{t}}{\ell}(n) }{2}, & \text{ if } m\not\equiv 0\pmod{2^k};\\
        \frac{(\bc{\infw{t}}{k}(m)-1).\bc{\infw{t}}{\ell}(n)}{2}+1, & \text{ if }m\equiv 0\pmod{2^k}.\\
      \end{array}\right.
  \]
In particular, for fixed $k\ge 1$ and $\ell\ge2$, the function $\bc{\infw{T}}{k,\ell}(m,n)$ is ultimately periodic in each coordinate, with periods $2^k$ and $2^\ell$, respectively.  
  



\end{theorem}
To get the full picture, recall from \cref{thm:lejeune} that $\bc{\infw{t}}{k}(m)=\mathsf{p}_{\infw{t}}(m)$ (resp. $\bc{\infw{t}}{\ell}(n)=\mathsf{p}_{\infw{t}}(n)$) if $m<2^k$ (resp. $n<2^\ell$).
\begin{example}
In \cref{tab:b23}, we have computed the first values of $\bc{\infw{T}}{2,3}(m,n)$. We can see four regions. The upper-left region is the product of factor complexities divided by~$2$. The upper-right region repeats periodically to the right. The lower-left region repeats periodically to the bottom. Finally, the lower-right region repeats in both directions. In particular, the last row and column show the beginning of a new period.
  \begin{table}[h!t]
\[
  \begin{array}{c|ccccccc|cccccccccc}
&1&2&3&4&5&6&7&8&9&10&11&12&13&14&15&16&\cdots\\
    \hline
1 &2 & 4 & 6 & 10 & 12 & 16 & 20 & 21 & 20 & 20 & 20 & 20 & 20 & 20 & 20 & 21\\
2 &4 & 8 & 12 & 20 & 24 & 32 & 40 & 42 & 40 & 40 & 40 & 40 & 40 & 40 & 40 & 42\\
3 & 6 & 12 & 18 & 30 & 36 & 48 & 60 & 63 & 60 & 60 & 60 & 60 & 60 & 60 & 60 & 63\\
  \hline
4 &9 & 17 & 25 & 41 & 49 & 65 & 81 & 85 & 81 & 81 & 81 & 81 & 81 & 81 & 81 & 85 \\
5 &8 & 16 & 24 & 40 & 48 & 64 & 80 & 84 & 80 & 80 & 80 & 80 & 80 & 80 & 80 & 84 \\
6 &8 & 16 & 24 & 40 & 48 & 64 & 80 & 84 & 80 & 80 & 80 & 80 & 80 & 80 & 80 & 84 \\
7 &8 & 16 & 24 & 40 & 48 & 64 & 80 & 84 & 80 & 80 & 80 & 80 & 80 & 80 & 80 & 84 \\
    8 &9 & 17 & 25 & 41 & 49 & 65 & 81 & 85 & 81 & 81 & 81 & 81 & 81 & 81 & 81 & 85 \\
    \vdots & &&&&&&&&&&&&&&&&\ddots \\
\end{array}
\]
 
  \caption{The first values of $\bc{\infw{T}}{2,3}(m,n)$.}
  \label{tab:b23}
\end{table}
\end{example}


Before proceeding to the proof of Theorem~\ref{thm:mainTM}, two auxiliary results are needed.

\begin{proposition}\label{pro:reconstruction}
Let $k\ge 1$, $\ell\ge 2$, $m\ge k$, $n\ge\ell$, and let
${u}\in\Fac_{m}(\infw{t})$, ${v}\in\Fac_{n}(\infw{t})$ with
${u}\not\sim_{k}\overline{{u}}$. Set ${W}={u}\oplus {v}$.
\begin{enumerate}[label={\upshape(\roman*)}]
\item The class $[{W}]_{k,\ell}$ determines the pair
  $([{u}]_{k},[{v}]_{\ell})$ up to the simultaneous bit-wise complement 
  $([{u}]_{k},[{v}]_{\ell})\leftrightarrow ([\overline{{u}}]_{k},[\overline{{v}}]_{\ell})$.
\item If moreover ${v}\not\sim_{\ell}\overline{{v}}$, then
  ${u}\oplus {v}\not\sim_{k,\ell}{u}\oplus\overline{{v}}$.
\end{enumerate}
\end{proposition}

\begin{proof}  
  Since $u\not\sim_k \overline{u}$, there exists a word $w\in\{0,1\}^k$ such that $$\alpha=\binom{u}{w}\neq\binom{\overline{u}}{w}=\binom{u}{\overline{w}}=\beta.$$

\noindent  
\underline{Part 1.} We first show that the
unordered pair $\{\alpha,\beta\}$ is determined by
$[{W}]_{k,\ell}$.

Let $z=a_1\cdots a_\ell\in\{0,1\}^\ell$. We define the array $P^{(z)}=w\oplus z$, i.e., the $j^{\text{th}}$ column of $P^{(z)}$ is $w$ (resp. $\overline{w}$) if $a_j=0$ (resp. $a_j=1$).
 Let $x$ be a word and $a\in\{0,1\}$. We denote by $\overline{x}^a$ the word
  \[
    \overline{x}^a=\left\{
      \begin{array}{l}
        x, \text{ if }a=0;\\
        \overline{x}, \text{ if }a=1.\\
      \end{array}\right.
    \]   
 Proceeding as in the proof of \cref{thm:bounded_complexity}, we get, as in \eqref{eq:develop}, 
  \begin{align}
    \binom{u\oplus v}{P^{(z)}}&=\sum_{x=b_1\cdots b_\ell \in\{0,1\}^\ell}\binom{v}{x} \binom{\overline{u}^{b_1}}{\overline{w}^{a_1}}\cdots\binom{\overline{u}^{b_\ell}}{\overline{w}^{a_\ell}} \nonumber\\
                              &=\sum_{x=b_1\cdots b_\ell \in\{0,1\}^\ell}\binom{v}{x} \alpha^{\sum_{j=1}^\ell \delta_{a_j,b_j}} \beta^{\ell-\sum_{j=1}^\ell \delta_{a_j,b_j}} \label{eq:asin}
  \end{align}
because $\binom{\overline{u}^{b_j}}{\overline{w}^{a_j}}=\binom{u}{w}=\alpha$ if $a_j=b_j$ and $\binom{\overline{u}^{b_j}}{\overline{w}^{a_j}}=\binom{u}{\overline{w}}=\beta$ if $a_j=1-b_j$.
  
  Let $z_1<z_2<\cdots <z_{2^\ell}$ be the words of $\{0,1\}^\ell$ genealogically ordered. The above relation shows that there is a $2^\ell\times 2^\ell$ matrix $M$ such that
  \begin{equation}\label{eq:kronecker_fold}
    M
    \begin{pmatrix}
      \binom{v}{z_1}\\
      \vdots\\
      \binom{v}{z_{2^\ell}}\\
    \end{pmatrix}=
    \begin{pmatrix}
      \binom{u\oplus v}{P^{(z_1)}}\\
      \vdots\\
      \binom{u\oplus v}{P^{(z_{2^\ell})}}\\
      \end{pmatrix}=
    \begin{pmatrix}
      \binom{W}{P^{(z_1)}}\\
      \vdots\\
      \binom{W}{P^{(z_{2^\ell})}}\\
      \end{pmatrix}.
    \end{equation}
    Thanks to \cref{lem:hamming}, $M$ is the $\ell$-fold Kronecker product of
    \[
      \begin{pmatrix}
  \alpha & \beta \\
  \beta & \alpha \\
\end{pmatrix}.
\]
Hence $\det M=(\alpha^2-\beta^2)^{\ell 2^{\ell-1}}$ and $M$ is invertible because $\alpha\neq\beta$ (more precisely, $\alpha-\beta$ and $\alpha+\beta$ are non-zero).

Since $M$ depends on $\alpha$ and~$\beta$, we are now ready to show that the
unordered pair $\{\alpha,\beta\}$ is determined by
$[{W}]_{k,\ell}$. Set
\[
  A:=\binom{{W}}{{w}},\quad
  B:=\binom{{W}}{\overline{{w}}},\quad
  D:=
    \binom{{W}}{\begin{pmatrix}{w}&{w}\end{pmatrix}},\quad
  E:=
    \binom{{W}}{\begin{pmatrix}\overline{{w}}&\overline{{w}}\end{pmatrix}}.
\]
These are binomial coefficients of~${W}$ with patterns of size at
most $k\times 2$, hence determined by
$[{W}]_{k,\ell}$. Recalling that $|v|=n=|v|_0+|v|_1$, 
the expansion~\eqref{eq:develop} with $\ell=1$ gives the two linear relations 
\begin{equation}
  \label{eq:linrel}
  A=|v|_0\, \alpha+|v|_1\, \beta \quad\text{ and }\quad B=|v|_1\, \alpha+|v|_0\, \beta.
\end{equation}
Adding these two equations determines $\alpha+\beta$
\[
  \alpha+\beta=\frac{A+B}{n}.
\]
To recover the remaining symmetric information without assuming prior knowledge of $|v|_0$ and $|v|_1$, we also use patterns of width two. 
The expansion~\eqref{eq:develop} with $\ell=2$ gives
$D=\binom{|v|_0}{2}\alpha^{2}+ \binom{|v|_1}{2}\beta^{2}+|v|_0\, |v|_1\,\alpha\beta$. 
Moreover, expanding $(|v|_0\, \alpha+|v|_1\, \beta)^{2}=A^{2}$ and comparing
with $2D$ yields
$$|v|_0\, \alpha^{2}+|v|_1\, \beta^{2}=A^{2}-2D;$$
the same computation with
$\overline{{w}}$ in place of~${w}$ gives
$|v|_1\, \alpha^{2}+|v|_0\, \beta^{2}=B^{2}-2E$. Adding these two identities,
\[
  \alpha^{2}+\beta^{2}
  =\frac{A^{2}+B^{2}-2D-2E}{n},
\]
and therefore $\alpha+\beta$ and 
$\alpha\beta=\bigl[(\alpha+\beta)^{2}-(\alpha^{2}+\beta^{2})\bigr]/2$
are determined by $[{W}]_{k,\ell}$. Hence $\alpha$ and $\beta$ are
the roots of a known quadratic polynomial, and the unordered pair
$\{\alpha,\beta\}$ is determined.
\smallskip

\noindent
\underline{Part 2.} Since $|v|=n\geq\ell$, the equivalence class $[{v}]_{\ell}$ is determined by the components $\binom{v}{z_1},\ldots,\binom{v}{z_{2^\ell}}$ of the vector occurring in \eqref{eq:kronecker_fold}. Proceed to the same computation as \eqref{eq:asin} with $\overline{v}$:
 \begin{align*}
   \binom{u\oplus \overline{v}}{P^{(z)}}&=\sum_{x=b_1\cdots b_\ell \in\{0,1\}^\ell}\binom{v}{\overline{x}} \binom{\overline{u}^{b_1}}{\overline{w}^{a_1}}\cdots\binom{\overline{u}^{b_\ell}}{\overline{w}^{a_\ell}}\\
   &=\sum_{x=b_1\cdots b_\ell \in\{0,1\}^\ell}\binom{v}{x} \binom{\overline{u}^{(1-b_1)}}{\overline{w}^{a_1}}\cdots\binom{\overline{u}^{(1-b_\ell)}}{\overline{w}^{a_\ell}}\\
                              &=\sum_{x=b_1\cdots b_\ell \in\{0,1\}^\ell}\binom{v}{x} \beta^{\sum_{j=1}^\ell \delta_{a_j,b_j}} \alpha^{\ell-\sum_{j=1}^\ell \delta_{a_j,b_j}}
  \end{align*}
Hence swapping $\alpha$ and $\beta$ in~$M$ amounts to interchanging~${v}$
and~$\overline{{v}}$ in the vector recovered from the inversion of~$M$ in \eqref{eq:kronecker_fold}, i.e., 
$[W]_{k,\ell}$ determines $[{v}]_{\ell}$ up to bit-wise complement.
\smallskip

\noindent
\underline{Part 3.} Next, we show that once $[{W}]_{k,\ell}$ and $[{v}]_{\ell}$ are
fixed, $[{u}]_{k}$ is determined up to bit-wise complement.

The first part of the proof shows that $\{\alpha,\beta\}$, defined from a particular word~$w$, is determined by the class $[W]_{k,\ell}$. The fact that $\alpha \neq \beta$ is only useful for inverting the matrix $M$. We can observe that the approach used can actually be applied to any word $w$ of length $k$. We can thus conclude that, for all $w\in\{0,1\}^k$, the unordered pair
$$\left\{
X_w:=\binom{u}{w},\, X_{\overline{w}}:=\binom{u}{\overline{w}}
\right\}$$
is determined by the class $[W]_{k,\ell}$. 
If moreover $|v|_0\neq|v|_1$, then we have two linear relations (similar to \eqref{eq:linrel}) that are enough to recover exactly $X_w$ and $X_{\overline{w}}$.

So we will now assume that $|v|_0=|v|_1$, and thus only the unordered pair is determined. 
For all $w\in\{0,1\}^k$, we set
\[
  \Delta_{{w}}:= X_{{w}}-X_{\overline{{w}}}.
\]
To recover the $\Delta_{{w}}$'s up to a global sign, we use patterns
of width two. For ${w},{w}'\in\{{0},{1}\}^{k}$ and ${a},{b}\in\{{0},{1}\}$, let
\[
  C_{{a},{b}}({w},{w}')
  :=
  \binom{{u}\oplus {v}}
    {\begin{pmatrix}\overline{{w^\top}}^{\,{a}}&\overline{{w'}^\top}^{\,{b}}\end{pmatrix}}
  \]
where $w^\top$ and ${w'}^\top$ are viewed as columns to count a $k\times 2$ pattern occurrence. 
Applying~\eqref{eq:develop} with~$\ell=2$ gives
\begin{align*}
  C_{{0},{0}}({w},{w}')
  &=\binom{{v}}{{00}}X_{{w}}X_{{w}'}
   +\binom{{v}}{{01}}X_{{w}}X_{\overline{{w}'}}
   +\binom{{v}}{{10}}X_{\overline{{w}}}X_{{w}'}
   +\binom{{v}}{{11}}X_{\overline{{w}}}X_{\overline{{w}'}},\\
  C_{{0},{1}}({w},{w}')
  &=\binom{{v}}{{00}}X_{{w}}X_{\overline{{w}'}}
   +\binom{{v}}{{01}}X_{{w}}X_{{w}'}
   +\binom{{v}}{{10}}X_{\overline{{w}}}X_{\overline{{w}'}}
   +\binom{{v}}{{11}}X_{\overline{{w}}}X_{{w}'},\\
  C_{{1},{0}}({w},{w}')
  &=\binom{{v}}{{00}}X_{\overline{{w}}}X_{{w}'}
   +\binom{{v}}{{01}}X_{\overline{{w}}}X_{\overline{{w}'}}
   +\binom{{v}}{{10}}X_{{w}}X_{{w}'}
   +\binom{{v}}{{11}}X_{{w}}X_{\overline{{w}'}},\\
  C_{{1},{1}}({w},{w}')
  &=\binom{{v}}{{00}}X_{\overline{{w}}}X_{\overline{{w}'}}
   +\binom{{v}}{{01}}X_{\overline{{w}}}X_{{w}'}
   +\binom{{v}}{{10}}X_{{w}}X_{\overline{{w}'}}
   +\binom{{v}}{{11}}X_{{w}}X_{{w}'}.
\end{align*}
Hence
\begin{align*}
  &C_{{0},{0}}({w},{w}')
  +C_{{1},{1}}({w},{w}')
  -C_{{0},{1}}({w},{w}')
  -C_{{1},{0}}({w},{w}')\\
  &\qquad
  =\Bigl(\binom{{v}}{{00}}
        -\binom{{v}}{{01}}
        -\binom{{v}}{{10}}
        +\binom{{v}}{{11}}\Bigr)
   \,\Delta_{{w}}\,\Delta_{{w}'}.
\end{align*}
Since $|{v}|_{{0}}=|{v}|_{{1}}>0$, we have
\[
  \binom{{v}}{{00}}-\binom{{v}}{{01}}
  -\binom{{v}}{{10}}+\binom{{v}}{{11}}
  =\binom{|{{v}|}}{2}-2|{{v}|}_{{0}}^{2}
  =-|{v}|_{{0}}\neq 0.
\]
Therefore, for every ${w},{w}'\in\{{0},{1}\}^{k}$, the
product $\Delta_{{w}}\,\Delta_{{w}'}$ is determined by
$[{W}]_{k,\ell}$ and~$[{v}]_{\ell}$.

Choose ${w}_{0}\in\{{0},{1}\}^{k}$ such that
$\Delta_{{w}_{0}}\neq 0$; such a word exists because
${u}\not\sim_{k}\overline{{u}}$. Then the family
$(\Delta_{{w}})_{{w\in\{0,1\}^k}}$ is determined up to a common sign, since all
products $\Delta_{{w}}\,\Delta_{{w}_{0}}$ are known. Consequently,
\[
  X_{{w}}=\frac{X_w+X_{\overline{w}}+\Delta_{{w}}}{2},\qquad
  X_{\overline{{w}}}=\frac{X_w+X_{\overline{w}}-\Delta_{{w}}}{2}
\]
are determined up to the simultaneous exchange
$X_{{w}}\longleftrightarrow X_{\overline{{w}}}$ for all
${w}\in\{{0},{1}\}^{k}$, which is exactly the effect of
replacing ${u}$ by~$\overline{{u}}$.  Hence $[{u}]_{k}$ is determined up
to bit-wise complement in the balanced case as well.
\smallskip

\underline{Part 4}. We may now conclude the proof. 
Thus $[{W}]_{k,\ell}$ determines both $[{u}]_{k}$ and $[{v}]_{\ell}$ up
to bit-wise complement, leaving four candidate pairs:
$$([{u}]_{k},[{v}]_{\ell}),\ 
([\overline{{u}}]_{k},[\overline{{v}}]_{\ell}),\ 
([{u}]_{k},[\overline{{v}}]_{\ell}) \text{ and }
([\overline{{u}}]_{k},[{v}]_{\ell}).$$  The first two belong to the
same simultaneous-complement orbit and yield the class
$[{u}\oplus {v}]_{k,\ell}=[{W}]_{k,\ell}$, whereas the last two yield the
class $[{u}\oplus\overline{{v}}]_{k,\ell}$.

If $[{v}]_{\ell}=[\overline{{v}}]_{\ell}$, the two possibilities
coincide and assertion~(i) follows.  Assume now that
$[{v}]_{\ell}\neq[\overline{{v}}]_{\ell}$; we show that the two array
classes are distinct.  Applying~\eqref{eq:kronecker_fold} to both
${u}\oplus {v}$ and ${u}\oplus\overline{{v}}$ gives
\[
  M\begin{pmatrix}
      \binom{{v}}{{z}_{1}}-\binom{\overline{{v}}}{{z}_{1}}\\
      \vdots\\
      \binom{{v}}{{z}_{2^{\ell}}}-\binom{\overline{{v}}}{{z}_{2^{\ell}}}
    \end{pmatrix}
  =\begin{pmatrix}
      \binom{{u}\oplus {v}}{{P}^{({z}_{1})}}-\binom{{u}\oplus\overline{{v}}}{{P}^{({z}_{1})}}\\
      \vdots\\
      \binom{{u}\oplus {v}}{{P}^{({z}_{2^{\ell}})}}-\binom{{u}\oplus\overline{{v}}}{{P}^{({z}_{2^{\ell}})}}
    \end{pmatrix}.
\]
Since $n\geq\ell$, ${v}\not\sim_{\ell}\overline{{v}}$ implies that the left-hand vector is
nonzero; since $M$ is invertible, so is the right-hand one;
hence
$\binom{{u}\oplus {v}}{{P}^{({z}_{j})}}\neq\binom{{u}\oplus\overline{{v}}}{{P}^{({z}_{j})}}$
for some~$j$, and
${u}\oplus {v}\not\sim_{k,\ell}{u}\oplus\overline{{v}}$.  This proves
assertion~(ii) and completes the proof of assertion~(i).
\qed
\end{proof}


\begin{lemma}\label{lem:disjointness}
Let $k\ge 1$, $\ell\ge 2$, and let
${s},{u}\in\Fac_{m}(\infw{t})$,
${v},{v}'\in\Fac_{n}(\infw{t})$ with $m,n\ge 1$. If
${s}\sim_{k}\overline{{s}}$ and $u\not\sim_{k}s$, then
${u}\oplus {v}\not\sim_{k,\ell}{s}\oplus {v}'$.
\end{lemma}

\begin{proof}
Since ${s}\sim_{k}\overline{{s}}$, \cref{pro:counting_classes,rem:small_m}
gives $2^{k}\mid m$ and shows that $[{s}]_{k}$ is the unique
 $k$-binomial class in $\Fac_{m}(\infw{t})$ fixed by bit-wise complement; the
hypothesis $u\not\sim_{k}s$ therefore implies
$[{u}]_{k}\neq[\overline{{u}}]_{k}$.  Moreover $m\geq 2^{k}\geq k$, so
 there exists a word 
${w}\in\{{0},{1}\}^{k}$ with
$\Delta_{{w}}:=\binom{{u}}{{w}}-\binom{{u}}{\overline{{w}}}\neq 0$.
Note also that, by ${s}\sim_{k}\overline{{s}}$
and~\eqref{eq:bitwise},
$\binom{{s}}{{x}}=\binom{{s}}{\overline{{x}}}$ for every
${x}\in\{{0},{1}\}^{k}$.

If $|{v}|_{{0}}\neq|{v}|_{{1}}$, then, by 
 \eqref{eq:develop} with $\ell=1$,
$$\binom{{u}\oplus {v}}{{w}}-\binom{{u}\oplus {v}}{\overline{{w}}}
=(|{v}|_{{0}}-|{v}|_{{1}})\,\Delta_{{w}}\neq 0,$$
whereas
$\binom{{s}\oplus {v}'}{{w}}=\binom{{s}\oplus {v}'}{\overline{{w}}}$; hence
${u}\oplus {v}\not\sim_{k,\ell}{s}\oplus {v}'$.

If $|{v}|_{{0}}=|{v}|_{{1}}$, then, in the
notation of the proof of~\cref{pro:reconstruction}, the width-$2$
alternating sum
$C_{{0},{0}}({w},{w})
+C_{{1},{1}}({w},{w})
-C_{{0},{1}}({w},{w})
-C_{{1},{0}}({w},{w})$
equals $-|{v}|_{{0}}\,\Delta_{{w}}^{2}\neq 0$ for
${u}\oplus {v}$ and vanishes for ${s}\oplus {v}'$, the corresponding
difference for~${s}$ being zero; again
${u}\oplus {v}\not\sim_{k,\ell}{s}\oplus {v}'$. \qed
\end{proof}


\begin{proof}[of \cref{thm:mainTM}]
  We first treat the case $m\ge 2^{k}$ and $n\ge 2^{\ell}$.
  \cref{pro:reconstruction} states that the map $([u]_k,[v]_\ell)\mapsto [u\oplus v]_{k,\ell}$, restricted to pairs with $u\not\sim_k\overline{u}$, is at most $2$-to-$1$: If the pairs $([u]_k,[v]_\ell)$ and $([\overline{u}]_k,[\overline{v}]_\ell)$ are distinct, they are mapped to the same $(k,\ell)$-class and only these two pairs are mapped to this $(k,\ell)$-class.

\begin{itemize}
\item \underline{Case~1}: $m\not\equiv 0\pmod{2^{k}}$.\enspace
By~\cref{pro:counting_classes}, no $k$-binomial class on
$\Fac_{m}(\infw{t})$ is fixed by bit-wise complement, so
${u}\not\sim_{k}\overline{{u}}$ for every factor~${u}$.
\smallskip

\item \underline{Case~1.a}: $n\not\equiv 0\pmod{2^{\ell}}$.\enspace
By~\cref{pro:counting_classes},
${v}\not\sim_{\ell}\overline{{v}}$ for every factor~$v\in\Fac_n(\infw{t})$, so~\cref{pro:reconstruction}\,(ii)
gives ${u}\oplus {v}\not\sim_{k,\ell}{u}\oplus\overline{{v}}$. The map is
therefore exactly $2$-to-$1$:
the pairs $([u]_k,[v]_\ell)$ and $([\overline{u}]_k,[\overline{v}]_\ell)$ (resp. $([u]_k,[\overline{v}]_\ell)$ and $([\overline{u}]_k,[v]_\ell)$) are mapped to $[u\oplus v]_{k,\ell}$ (resp. $[u\oplus \overline{v}]_{k,\ell}$). Thus, 
$$\bc{\infw{T}}{k,\ell}(m,n)=\bc{\infw{t}}{k}(m)\cdot\bc{\infw{t}}{\ell}(n)/2.$$
\Cref{fig:case1a} illustrates the situation schematically: dots represent pairs of equivalence
classes $([{u}]_{k},[{v}]_{\ell})$, and edges connect pairs that map
to the same $(k,\ell)$-equivalence class of~$\infw{T}$.

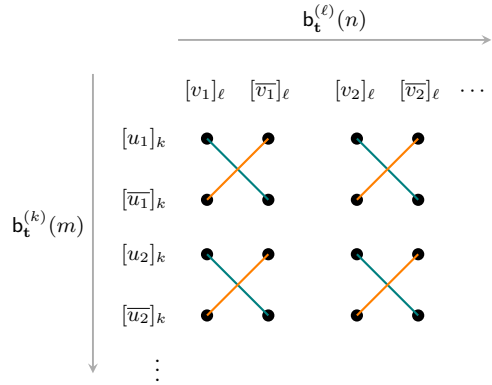
\begin{figure}[h!t]
\centering
\scalebox{0.9}{%
\begin{tikzpicture}[
    >=stealth,
    every node/.style={font=\small},
    blackdot/.style={circle, fill=black, inner sep=1.8pt},
    tealedge/.style={teal, line width=0.9pt},
    orangeedge/.style={orange, line width=0.9pt},
  ]
  \newcommand{\xblock}[2]{%
    \node[blackdot] at (#1-0.45,#2+0.45) {};%
    \node[blackdot] at (#1+0.45,#2+0.45) {};%
    \node[blackdot] at (#1-0.45,#2-0.45) {};%
    \node[blackdot] at (#1+0.45,#2-0.45) {};%
    \draw[tealedge] (#1-0.45,#2+0.45) -- (#1+0.45,#2-0.45);%
    \draw[orangeedge] (#1+0.45,#2+0.45) -- (#1-0.45,#2-0.45);%
  }
  \draw[->, gray!65, thick] (0.85,1.00) -- (5.45,1.00)
        node[midway,above,text=black] {$\bc{\infw{t}}{\ell}(n)$};
  \draw[->, gray!65, thick] (-0.40,0.50) -- (-0.40,-3.90)
        node[midway,left,text=black]  {$\bc{\infw{t}}{k}(m)$};
  \node at (1.25,0.25) {$[{v}_{1}]_{\ell}$};
  \node at (2.20,0.25) {$[\overline{{v}_{1}}]_{\ell}$};
  \node at (3.45,0.25) {$[{v}_{2}]_{\ell}$};
  \node at (4.40,0.25) {$[\overline{{v}_{2}}]_{\ell}$};
  \node at (5.20,0.25) {$\cdots$};
  \node[left] at (0.80,-0.45) {$[{u}_{1}]_{k}$};
  \node[left] at (0.80,-1.35) {$[\overline{{u}_{1}}]_{k}$};
  \node[left] at (0.80,-2.15) {$[{u}_{2}]_{k}$};
  \node[left] at (0.80,-3.05) {$[\overline{{u}_{2}}]_{k}$};
  \node[left] at (0.70,-3.75) {$\vdots$};
  \xblock{1.725}{-0.9}
  \xblock{3.925}{-0.9}
  \xblock{1.725}{-2.6}
  \xblock{3.925}{-2.6}
\end{tikzpicture}%
}
\caption{Equivalence classes for $\sim_{k,\ell}$ when
$m\not\equiv 0\pmod{2^{k}}$ and $n\not\equiv 0\pmod{2^{\ell}}$.}
\label{fig:case1a}
\end{figure}
\smallskip
\item \underline{Case~1.b:} $n\equiv 0\pmod{2^{\ell}}$.\enspace
By~\cref{pro:counting_classes}, there exists a unique
$\ell$-binomial class of length-$n$ factors of~$\infw{t}$ that is
fixed by bit-wise complement; let ${s}\in\Fac_{n}(\infw{t})$ be any
representative, so that ${s}\sim_{\ell}\overline{{s}}$.
Applying~\eqref{eq:bitwise} and using
$\binom{{s}}{{w}}=\binom{\overline{{s}}}{{w}}$ for every
${w}\in\{{0},{1}\}^{\le\ell}$ shows that
${u}\oplus {s}\sim_{k,\ell}\overline{{u}}\oplus {s}$, so the pairs
$([{u}]_{k},[{s}]_{\ell})$ and $([\overline{{u}}]_{k},[{s}]_{\ell})$ yield
the same $(k,\ell)$-class. There are $\frac{\bc{\infw{t}}{k}(m)}{2}$ such $(k,\ell)$-classes.

For every non-fixed class $[{v}]_{\ell}$ (that is, with
${v}\not\sim_{\ell}\overline{{v}}$ and there are $\bc{\infw{t}}{\ell}(n)-1$ such classes), we use 
\cref{pro:reconstruction}\,(ii) as in the previous sub-case. In this situation, we therefore obtain 
$\bc{\infw{t}}{k}(m)\bigl(\bc{\infw{t}}{\ell}(n)-1\bigr)/2$ such $(k,\ell)$-classes. 
The total sum equals
\[
  \frac{\bc{\infw{t}}{k}(m)}{2}
  +\frac{\bc{\infw{t}}{k}(m)\cdot \bigl(\bc{\infw{t}}{\ell}(n)-1\bigr)}{2}
  =\frac{\bc{\infw{t}}{k}(m)\cdot\bc{\infw{t}}{\ell}(n)}{2},
\]
again confirming the Case~1 formula. \Cref{fig:case1b} offers a
schematic illustration.

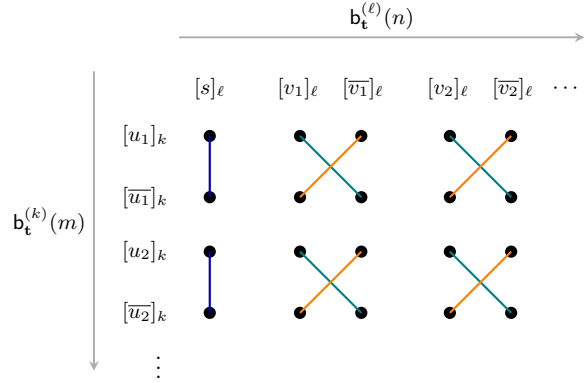
\begin{figure}[h!t]
\centering
\scalebox{0.9}{%
\begin{tikzpicture}[
    >=stealth,
    every node/.style={font=\small},
    blackdot/.style={circle, fill=black, inner sep=1.8pt},
    blueedge/.style={blue!70!black, line width=0.9pt},
    tealedge/.style={teal, line width=0.9pt},
    orangeedge/.style={orange, line width=0.9pt},
  ]
  \newcommand{\xblock}[2]{%
    \node[blackdot] at (#1-0.45,#2+0.45) {};%
    \node[blackdot] at (#1+0.45,#2+0.45) {};%
    \node[blackdot] at (#1-0.45,#2-0.45) {};%
    \node[blackdot] at (#1+0.45,#2-0.45) {};%
    \draw[tealedge] (#1-0.45,#2+0.45) -- (#1+0.45,#2-0.45);%
    \draw[orangeedge] (#1+0.45,#2+0.45) -- (#1-0.45,#2-0.45);%
  }
  \draw[->, gray!65, thick] (0.85,1.00) -- (6.80,1.00)
        node[midway,above,text=black] {$\bc{\infw{t}}{\ell}(n)$};
  \draw[->, gray!65, thick] (-0.40,0.50) -- (-0.40,-3.90)
        node[midway,left,text=black]  {$\bc{\infw{t}}{k}(m)$};
  \node at (1.30,0.25) {$[{s}]_{\ell}$};
  \node at (2.60,0.25) {$[{v}_{1}]_{\ell}$};
  \node at (3.55,0.25) {$[\overline{{v}_{1}}]_{\ell}$};
  \node at (4.80,0.25) {$[{v}_{2}]_{\ell}$};
  \node at (5.75,0.25) {$[\overline{{v}_{2}}]_{\ell}$};
  \node at (6.55,0.25) {$\cdots$};
  \node[left] at (0.80,-0.45) {$[{u}_{1}]_{k}$};
  \node[left] at (0.80,-1.35) {$[\overline{{u}_{1}}]_{k}$};
  \node[left] at (0.80,-2.15) {$[{u}_{2}]_{k}$};
  \node[left] at (0.80,-3.05) {$[\overline{{u}_{2}}]_{k}$};
  \node[left] at (0.70,-3.75) {$\vdots$};
  \foreach \yy in {-0.45,-1.35,-2.15,-3.05}
    {\node[blackdot] at (1.30,\yy) {};}
  \draw[blueedge] (1.30,-0.45) -- (1.30,-1.35);
  \draw[blueedge] (1.30,-2.15) -- (1.30,-3.05);
  \xblock{3.075}{-0.9}
  \xblock{5.275}{-0.9}
  \xblock{3.075}{-2.6}
  \xblock{5.275}{-2.6}
\end{tikzpicture}%
}
\caption{Equivalence classes for $\sim_{k,\ell}$ when
$m\not\equiv 0\pmod{2^{k}}$ and $n\equiv 0\pmod{2^{\ell}}$.}
\label{fig:case1b}
\end{figure}

\item \underline{Case~2:} $m\equiv 0\pmod{2^{k}}$.\enspace
By~\cref{pro:counting_classes}, there exists exactly one
$k$-binomial class of length-$m$ factors of~$\infw{t}$ that is
fixed by bit-wise complement. Let ${s}\in\Fac_{m}(\infw{t})$ be a
representative, so that ${s}\sim_{k}\overline{{s}}$. For every column
${P}_{j}$ of a $k\times\ell$ pattern~${P}=
\begin{pmatrix}
  P_1&\cdots&P_\ell
\end{pmatrix}
$,
\[
  \binom{{s}}{{P}_{j}}
  =\binom{\overline{{s}}}{{P}_{j}}
  =\binom{{s}}{\overline{{P}_{j}}},
\]
the first equality by ${s}\sim_{k}\overline{{s}}$ and the second
by~\eqref{eq:bitwise}.  The
expansion~\eqref{eq:develop} then gives
\[
  \binom{{s}\oplus {v}}{{P}}=\sum_{w\in\{0,1\}^\ell}\binom{v}{w}\binom{s}{\overline{P_1}^{w_1}}\cdots \binom{s}{\overline{P_\ell}^{w_\ell}}
  =\binom{n}{\ell}\prod_{j=1}^{\ell}\binom{{s}}{{P}_{j}},
\]
which is independent of~${v}$. Thus all pairs
$([{s}]_{k},[{v}]_{\ell})$ yield a single $(k,\ell)$-equivalence
class, which we denote $[{s}\oplus {v}]_{k,\ell}$.

For every $[{u}]_{k}\neq[{s}]_{k}$, the class $[{u}]_{k}$ is not fixed
by bit-wise complement, and~\cref{pro:reconstruction} applies: the
restriction of the map
$([{u}]_{k},[{v}]_{\ell})\mapsto[{u}\oplus {v}]_{k,\ell}$ to
$[{u}]_{k}\neq[{s}]_{k}$ is exactly $2$-to-$1$.
By~\cref{lem:disjointness}, the class $[{s}\oplus {v}]_{k,\ell}$ is
distinct from every class arising from
$([{u}]_{k},[{v}]_{\ell})$ with $[{u}]_{k}\neq[{s}]_{k}$.
This accounts for
\[
  \frac{\bigl(\bc{\infw{t}}{k}(m)-1\bigr)\cdot\bc{\infw{t}}{\ell}(n)}{2}
\]
distinct $(k,\ell)$-classes, to which we add the single class
$[{s}\oplus {v}]_{k,\ell}$ described above. In total,
\[
  \bc{\infw{T}}{k,\ell}(m,n)
  =\frac{\bigl(\bc{\infw{t}}{k}(m)-1\bigr)\cdot\bc{\infw{t}}{\ell}(n)}{2}+1,
\]
as claimed. \Cref{fig:case2} illustrates the situation
schematically.

\begin{figure}[h!t]
\centering
\scalebox{0.9}{%
\begin{tikzpicture}[
    >=stealth,
    every node/.style={font=\small},
    blackdot/.style={circle, fill=black, inner sep=1.8pt},
    blueedge/.style={blue!70!black, line width=0.9pt},
    tealedge/.style={teal, line width=0.9pt},
    orangeedge/.style={orange, line width=0.9pt},
  ]
  \newcommand{\xblock}[2]{%
    \node[blackdot] at (#1-0.45,#2+0.45) {};%
    \node[blackdot] at (#1+0.45,#2+0.45) {};%
    \node[blackdot] at (#1-0.45,#2-0.45) {};%
    \node[blackdot] at (#1+0.45,#2-0.45) {};%
    \draw[tealedge] (#1-0.45,#2+0.45) -- (#1+0.45,#2-0.45);%
    \draw[orangeedge] (#1+0.45,#2+0.45) -- (#1-0.45,#2-0.45);%
  }
  \draw[->, gray!65, thick] (2.30, 1.70) -- (7.00, 1.70)
        node[midway,above,text=black] {$\bc{\infw{t}}{\ell}(n)$};
  \draw[->, gray!65, thick] (1.00, 0.85) -- (1.00,-3.90)
        node[midway,left,text=black]  {$\bc{\infw{t}}{k}(m)$};
  \node at (2.60,1.20) {$[{v}_{1}]_{\ell}$};
  \node at (3.55,1.20) {$[\overline{{v}_{1}}]_{\ell}$};
  \node at (4.80,1.20) {$[{v}_{2}]_{\ell}$};
  \node at (5.75,1.20) {$[\overline{{v}_{2}}]_{\ell}$};
  \node at (6.55,1.20) {$\cdots$};
  \node[left] at (2.20, 0.45) {$[{s}]_{k}$};
  \node[left] at (2.20,-0.45) {$[{u}_{1}]_{k}$};
  \node[left] at (2.20,-1.35) {$[\overline{{u}_{1}}]_{k}$};
  \node[left] at (2.20,-2.15) {$[{u}_{2}]_{k}$};
  \node[left] at (2.20,-3.05) {$[\overline{{u}_{2}}]_{k}$};
  \node[left] at (2.10,-3.75) {$\vdots$};
  \fill[red!15, rounded corners=3pt] (2.30, 0.65) rectangle (6.80, 0.15);
  \foreach \xpos in {2.55,3.45,4.75,5.65}
    {\node[blackdot] at (\xpos, 0.40) {};}
  \xblock{3.075}{-0.9}
  \xblock{5.275}{-0.9}
  \xblock{3.075}{-2.6}
  \xblock{5.275}{-2.6}
\end{tikzpicture}%
}
\caption{Equivalence classes for $\sim_{k,\ell}$ when
$m\equiv 0\pmod{2^{k}}$ and $n\not\equiv 0\pmod{2^{\ell}}$.}
\label{fig:case2}
\end{figure}
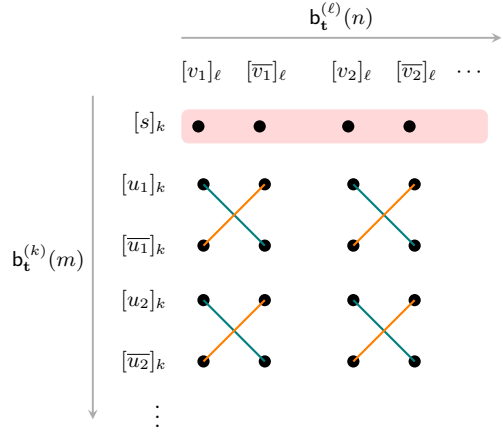

\end{itemize}


It remains to consider the cases where $m<2^k$ or $n<2^\ell$.
If $n\ge 2$, put $r=\min(k,m)$ and $s=\min(\ell,n)$, so that $s\ge 2$.
On arrays of size $m\times n$, the relations $\sim_{k,\ell}$ and $\sim_{r,s}$ coincide;
likewise, the relevant word classes are unchanged when $k,\ell$ are replaced by $r,s$.
Thus \cref{pro:reconstruction} applies with parameters $r,s$.
By \cref{pro:counting_classes,rem:small_m}, the original vertical quotient has
exactly one class fixed by complement if $2^k\mid m$, and none otherwise;
the analogous statement holds for the horizontal quotient with $2^\ell\mid n$.
The same counting argument, using \cref{lem:kkequiv,lem:disjointness} for the fixed vertical class, therefore applies.
Finally, if $n=1$, one-column arrays are words, so
$\bc{\infw{T}}{k,\ell}(m,1)=\bc{\infw{t}}{k}(m)$.
Since $\bc{\infw{t}}{\ell}(1)=2$, both branches of the stated formula give this value.
\qed
\end{proof}

\begin{remark}
  We may observe an asymmetry in the formula $\frac{(\bc{\infw{t}}{k}(m)-1).\bc{\infw{t}}{\ell}(n)}{2}+1$ in \cref{thm:mainTM}, \cref{fig:case1b} and \cref{fig:case2}, this comes from the choice of the ordering in \cref{def:main} of the binomial coefficients where columns are privileged. 
\end{remark}

\section{Concluding remarks}

\cref{thm:mainTM} is a showcase for the computation of the $(k,\ell)$-binomial complexity of an infinite array of genuine interest. It is enlightening to see how the $(k,\ell)$-binomial equivalence classes are organized with respect to the $k$-binomial equivalence classes of the Thue--Morse factors. The arguments of \cref{thm:mainTM} also extend to XOR-products of binary words whose factor languages are closed under bit-wise complement. However, a generalization of this result to Thue--Morse arrays over $m$ letters is not straightforward. Indeed, the formulas and results about bit-wise complement do not seem to extend in a direct way. The exact counting of $\sim_{k,\ell}$-equivalence classes remains therefore challenging even though we conjecture a periodic behavior. For the array $\infw{T}_m$ defined by $\infw{T}_m[i,j]=\infw{t}_m[i]+\infw{t}_m[j]\bmod m$, \cref{thm:bounded_complexity} gives the bound $\bc{\infw{T}_m}{k,k}(n,n)\le \bigl(\bc{\infw{t}_m}{k}(n)\bigr)^2$. In the ternary case, $([u]_k,[v]_k)$, $([u+1]_k,[v+2]_k)$ and $([u+2]_k,[v+1]_k)$ yield the same class $[u\oplus v]_{k,k}$, where addition is considered component-wise and modulo~$3$; however, these three pairs of classes need not be distinct.

Sturmian words have been characterized by their $k$-binomial complexity in \cite{FiciPuzynina,RigoSalimov-2015,RigoStipulantiWhiteland-2024}. We recall the result: If the infinite word $\infw{x}$ is Sturmian then, for all $k\ge 2$, $\bc{\infw{x}}{k}(n)=n+1$ for all $n\ge 0$. Conversely, if there exists some $k\ge 2$ such that $\bc{\infw{x}}{k}(n)=n+1$ for all $n\ge 0$, then $\infw{x}$ is Sturmian. From \cref{cor:sturmian}, we know that if $\infw{x}$ and $\infw{y}$ are two Sturmian words, then for all $m,n\ge 1$, we have
$\bc{\infw{x}\times\infw{y}}{k,\ell}(m,n)= (m+1)(n+1)$  whenever $k,\ell\ge 2$. Can we characterize the family of infinite arrays with that property ? Could we, in particular, as investigated in one dimension \cite{Vivion2025}, characterize families of arrays whose binomial complexity is equal to factor complexity? For instance, the two-dimensional generalization of Sturmian words corresponding to an approximation of a plane introduced in \cite{BertheVuillon} is worth investigating. 

In the problem addressed above, if we allow $k=\ell=1$, then one can easily find an array~$\infw{C}$ such that $\bc{\infw{C}}{1,1}(m,n)=(m+1)(n+1)$ for all $m,n\ge 1$. 
Let $\infw{c}$ be the (binary) Champernowne word $011011100101110111\cdots$ concatenating the binary representations of the nonnegative integers in order. Since $\infw{c}$ contains both $0^n$ and $1^n$, its abelian complexity is $n+1$. By \cref{pro:complex_prod}, $\bc{\infw{c}\times\infw{c}}{1,1}(m,n)= (m+1)(n+1)$. However,  $\bc{\infw{c}\times\infw{c}}{1,1}(m,n)<\bc{\infw{c}\times\infw{c}}{k,\ell}(m,n)$ when $k>1$ and $m\ge 2$ or, $\ell>1$ and  $n\ge 2$.

\section*{Acknowledgments}

 In a final stage, we ran ChatGPT Astra to correct several minor errors.

\bibliographystyle{plainurl}
\bibliography{./bibliography}

\end{document}